\documentclass[11pt]{article}

\usepackage[sort]{cite}       % nicer citations
\usepackage{amssymb}          % various ams stuff
\usepackage{latexsym}         % dito
\usepackage{amsmath}          % dito
\usepackage[latin1]{inputenc} % smart input of special characters
\usepackage{eucal}            % nicer caligraphic fonts
\usepackage{bbm}              % nicer bbm fonts
\usepackage{theorem}          % nettere Theoreme
\usepackage{stmaryrd}         % noch ein paar mehr Symbole (Klammern)
\usepackage[english]{babel}

\usepackage[all]{xy}
\usepackage[ps]{xypic}

\usepackage{lscape}

\title{Deformation Theory of Courant Algebroids via the Rothstein Algebra}

\author{\textbf{Frank Keller}\thanks{E-mail: F.Keller@impan.pl}
  \\[0.1cm]
  Institute of Mathematics\\
  Polish Academy of Sciences\\
  ul. {\'S}niadeckich 8\\
  00-956 Warszawa 10\\
  Poland\\[0.5cm]
  \addtocounter{footnote}{5}
  \textbf{Stefan Waldmann}\thanks{E-mail: Stefan.Waldmann@physik.uni-freiburg.de}
  \\[0.1cm]
  Fakult{\"a}t f{\"u}r Mathematik und Physik\\
  Albert-Ludwigs-Universit{\"a}t Freiburg\\
  Physikalisches Institut\\
  Hermann Herder Stra{\ss}e 3\\
  D 79104 Freiburg\\
  Germany}

\date{August 2011}

\renewcommand{\mathbb}[1]{\mathbbm{#1}} % use nicer bbm fonts

\newcounter{comment}

\newcommand{\id}         {\operatorname{\mathsf{id}}}   
 
\newcommand{\sign}       {\operatorname{\mathrm{sign}}}

\newcommand{\ad}         {\operatorname{\mathrm{ad}}}    
    
\newcommand{\Hom}        {\operatorname{\mathsf{Hom}}}   
\newcommand{\End}        {\operatorname{\mathsf{End}}}   
\newcommand{\SP}[1]      {\left\langle{#1}\right\rangle} 
\newcommand{\ring}[1]    {\mathsf{#1}}

\newcommand{\Der}        {\operatorname{\mathsf{Der}}}

\newcommand{\Sym}        {\operatorname{\mathrm{S}}}
\newcommand{\Anti}       {\operatorname{\Lambda}}
\newcommand{\I}          {\mathrm{i}}
\newcommand{\E}          {\mathrm{e}}

\newcommand{\SDer}       {\operatorname{\mathsf{Der}_{\mathrm{sym}}}}
\newcommand{\Rothstein}  {\operatorname{\mathcal{R}}}
\newcommand{\RothBracket}[1]  {\left\{#1\right\}_{\mathrm{R}}}

\newcommand{\A}{\mathcal{A}}
\newcommand{\C}{\mathsf{C}}
\newcommand{\Eps}{\mathcal{E}}
\newcommand{\kk}{\mathbb k}
\newcommand{\MC}{\mathcal{C}}
\newcommand{\Comega}{\Omega_{\mathcal{C}}}

\newcommand{\dermod}{\mathfrak{D}}
\newcommand{\W}{\operatorname{\mathcal{W}}}
\newcommand{\degs}{\deg_s}
\newcommand{\dega}{\deg_a}
\newcommand{\degE}{\deg_\Eps}

\newcommand{\Deg}{\operatorname{\mathsf{Deg}}}
\newcommand{\DegR}{\operatorname{\mathsf{Deg}}_\mathcal{R}}
\newcommand{\parity}{\operatorname{par}}
\newcommand{\fedosovD}{\mathcal{D}}
\newcommand{\switch}[1]{\operatorname{\tau}_{#1}}

\newcommand{\calQ}{\operatorname{\mathcal{Q}}}
\newcommand{\calS}{\operatorname{\mathcal{S}}}

\newcommand{\LL}{\mathcal{L}}
\newcommand{\LplusL}{\mathcal{L\oplus L'}}
\newcommand{\CM}{\mathcal{M}_{\mathrm{CM}}}

\newcommand{\elide}[1]{\stackrel{#1}{\Hat{\ldots}}}
\newcommand{\elidetwo}[2]{\stackrel{#1}{\Hat{\ldots}}\,
  \stackrel{#2}{\Hat{\ldots}}\:\!}

\newcommand{\dif}{\operatorname{d}}
\newcommand{\shuffle}[1]{\mathcal{S}_{#1}}

\newtheorem{lemma}{Lemma}[section]
\newtheorem{proposition} [lemma] {Proposition}

\newtheorem{theorem} [lemma] {Theorem}
\newtheorem{corollary} [lemma] {Corollary}
\newtheorem{definition}[lemma] {Definition}

\theorembodyfont{\rmfamily}
\newtheorem{example}[lemma]{Example}
\newtheorem{remark}[lemma]{Remark}

\newenvironment{proof}[1][{}]{\par\noindent\textsc{Proof{#1}: }}{\hspace*{\fill}$\blacksquare$\smallskip\noindent\par}

\numberwithin{equation}{section}

\allowdisplaybreaks

\begin{document}

\maketitle

\begin{abstract}
    In this paper we define Courant algebroids in a purely algebraic
    way and study their deformation theory by using two different but
    equivalent graded Poisson algebras of degree $-2$. First steps
    towards a quantization of Courant algebroids are proposed by
    employing a Fedosov like deformation quantization.
\end{abstract}

\newpage

%
% table of contents
%

\tableofcontents

%
% Introduction
%

\section{Introduction}
\label{sec:Intro}

In differential geometry, a Courant algebroid is vector bundle
together with a not necessarily positive definite non-degenerate fiber
metric, with a bracket on the sections of the bundle, and with a
bundle map into the tangent bundle of the base manifold such that
certain compatibility conditions are fulfilled. This definition is a
generalization of the canonical Courant algebroid structure on $TM
\oplus T^\ast M$ which was introduced in \cite{courant:1990a}.
Courant algebroids have increasingly gained interest in various
contexts like generalized geometries, see e.g. \cite{gualtieri:2003a},
reduction, see e.g. \cite{bursztyn.cavalcanti.gualtieri:2007a}, and
symplectic super-manifolds, see e.g.  \cite{roytenberg:2002b,
  roytenberg:1999a, roytenberg.weinstein:1998a}, to mention only a few
instances.

Since Courant algebroids can be viewed to go beyond the structure of a
Lie algebroid and since for Lie algebroids there is an entirely
algebraic notion, namely that of Lie-Rinehart algebra (or Lie-Rinehart
pair), see \cite{rinehart:1963a, huebschmann:1998a}, it seems
desirable to cast the theory of Courant algebroids also in a purely
algebraic framework. Partially, this has been achieved in the language
of super-manifolds by Roytenberg in \cite{roytenberg:2002b,
  roytenberg:1999a}. One reason for such a purely algebraic
formulation is the fact that in the reduction procedures of Courant
algebroids one typically obtains quotients which will be no longer
smooth manifolds. Depending on the precise situation, orbifold
singularities or even ``worse'' can appear. Thus, an algebraic
framework is necessary in order to speak of Courant algebroids also in
this context.

Moreover, the deformation theory of Lie algebroids as established by
Crainic and Moerdijk in \cite{crainic.moerdijk:2004a} has a completely
algebraic formulation not referring to the underlying geometry. Thus a
deformation theory for Courant algebroids along the lines of
\cite{crainic.moerdijk:2004a} should be possible. A super-manifold
approach to deformation theory can be found in \cite{roytenberg:2002b}
using the language of derived brackets, see e.g.
\cite{kosmann-schwarzbach:2004b}. Here the deformation theory is still
purely classical: Courant structures should be deformed into Courant
algebroids. However, the lessons from various formality theorems like
\cite{kontsevich:2003a} tell that the classical deformation theory is
intimately linked to the deformation theory in the sense of
quantizations. While we will provide some first steps into this
direction, the latter has not been achieved in a satisfying way
yet. Nevertheless, we believe that a good understanding of the
classical deformation theory will provide guidelines for the
quantizations.

In this work we propose a definition of a Courant algebroid in purely
algebraic terms for a commutative associative unital algebra
$\mathcal{A}$ and a finitely generated projective module $\mathcal{E}$
with a full and strongly non-degenerate inner product. Our definition
comes along with a complex controlling the deformation theory, very
much in the spirit of \cite{crainic.moerdijk:2004a}. We show that this
complex is actually a graded Poisson algebra of degree $-2$ and
elements $m$ of degree $3$ with $[m, m] = 0$ correspond to Courant
algebroid structures on the pair $(\mathcal{A}, \mathcal{E})$. Having
this ambient Poisson algebra, the usual deformation theory based on
graded Lie algebras as established in \cite{gerstenhaber:1964a,
  nijenhuis.richardson:1967a} can easily be adapted to this situation.
In addition to this rather naturally defined complex we propose a
second, seemingly different complex based on the Rothstein algebra:
Here we use the symmetric algebra over the derivations of
$\mathcal{A}$ together with the Grassmann algebra over $\mathcal{E}$
and define a graded Poisson structure of degree $-2$ by means of a
connection. In the differential geometric context this is a particular
case of the Rothstein bracket \cite{rothstein:1991a} and has been used
already in \cite{keller.waldmann:2007a, keller:2004a} to discuss the
deformation theory of Dirac structures inside Courant algebroids.
Having this purely algebraic definition of the Rothstein algebra we
can establish a Poisson monomorphism which geometrically corresponds
to a symbol calculus for the maps in the before mentioned complex.
Under certain conditions, satisfied for smooth manifolds, this
monomorphism is even an isomorphism. In the general case, we can use
it to replace the bigger complex by the Rothstein algebra: this has
the advantage that certain ``unpleasant'' symmetric multiderivations
are shown to be irrelevant for the deformation problem.

Finally, the symbol calculus gives a derived bracket picture of
Courant algebroids also in the Rothstein approach. Since the Rothstein
picture seems to be much simpler it may be a good starting point for
any sort of deformation quantization of Courant algebroids based on
the use of a connection. To illustrate this we use a Fedosov-like
construction of a deformation quantization of the Rothstein algebra as
this has already been used by Bordemann \cite{bordemann:1996a} in a
differential geometric context, see \cite{fedosov:1996a} for the
original construction of Fedosov. Under mild assumptions on the
algebra $\mathcal{A}$ we can show the existence of a suitable
deformation quantization $\star$ of the ambient Rothstein algebra. We
propose two different definitions of the Courant structure: first we
ask for a deformation of the Courant structure into an element with
square zero with respect to the star product $\star$. It is then a
standard argument that the obstructions as well as the infinitesimals
for such a quantization are elements in the same cohomology groups as
for the classical deformation theory, supporting the importance of the
complex we constructed.  However, as already seen in simple examples
our first version seems to be a quite strong requirement. Hence we
relax the definition and ask only for a deformation with square being
central. This is shown to be possible in some first examples and still
yields an inner differential, thus matching the characteristic feature
of the classical version of a Courant algebroid.  As a large class of
examples we discuss Lie-Rinehart pairs viewed as particular Courant
algebroids and study their deformation quantization. In this case,
even the stronger version of a deformation quantization is possible.

The paper is organized as follows: In
Section~\ref{AlgebraicDefinitionCourantAlgebroids} we state the
algebraic definition of a Courant algebroid and specify the category
of such Courant algebroids. In Section~\ref{ThePoissonAlgebraC} the
complex $\mathcal{C}^\bullet(\mathcal{E})$ is defined and the graded
Poisson algebra structure on $\mathcal{C}^\bullet(\mathcal{E})$ is
constructed in detail. Section~\ref{TheRothsteinAlgebraR} contains the
construction of the Rothstein algebra while
Section~\ref{sec:TheIsomorphism} establishes the relation between the
two approaches. In Section~\ref{section:Fedosov} we adapt the
well-known Fedosov construction of star products to the cases of the
Rothstein
algebra. Section~\ref{section:QuantizationOfCourantAlgebroids}
contains the definition of a deformation quantization of Courant
structures inside the ambient quantization of the Rothstein algebra
together with first examples. Finally, in
Section~\ref{section:Lie-Rinehart-Pairs} we show how a Lie-Rinehart
pair can be viewed as a Courant algebroid. We show that the
deformation quantization is possible in this case.

\noindent
\textbf{Acknowledgments:} We would like to thank Johannes Huebschmann
for valuable suggestions.

%
% Algebraic Definition of Courant Algebroids
%

\section{Algebraic Definition of Courant Algebroids}
\label{AlgebraicDefinitionCourantAlgebroids}

Let $\A$ be a commutative, unital algebra over a ring
$\ring{R}\supseteq \mathbb{Q}$ and let $\Eps$ be a projective,
finitely generated module over $\A$.  Let further $\Eps$ be endowed
with a symmetric, strongly non-degenerate, full $\A$-bilinear form
\[
\SP{\cdot,\cdot}:\Eps \times \Eps \longrightarrow \A.
\]
Recall that a bilinear form is strongly non-degenerate if and only if
the induced map $\Eps \longrightarrow \Eps'= \Hom_\A(\Eps,\A)$ is an
isomorphism.  Moreover, a bilinear form is called full if and only if
the module homomorphism
\[
\phi:\Eps \otimes_\A \Eps \longrightarrow \A
\]
induced by $\SP{\cdot,\cdot}$ is surjective, i.e. if every $a\in \A$
can be written as a finite sum $a=\sum_i \SP{x_i,y_i}$ with
$x_i,y_i\in \Eps$. Such a bilinear form is also  called a full
$\mathcal{A}$-valued inner product.
\begin{remark} 
   As a consequence of fullness, our module $\Eps$ is also
   faithful, i.e. $a z = 0$ for all $z\in \Eps$ implies $a = 0$. 
   To see this, write the unit element $1\in\A$ as a finite sum $ 1 = 
   \sum_i \SP{x_i,y_i}$ with $x_i,y_i \in \Eps$. If  now
   $a\in \A$ with  $a z = 0$ $\forall\; z\in \Eps$, then it follows that
   \begin{equation}
        a = a\cdot 1
       = a \sum_i \SP{x_i,y_i}= \sum_i\SP{a x_i,y_i} = 0.
   \end{equation}
\end{remark}
\begin{remark}
    \label{remark:Morita}
    The existence of such a strongly non-degenerate and full inner
    product is a rather strong requirement. In fact, in the framework
    of $^*$-algebras over a ring $\ring{C} = \ring{R}(\I)$ with $\I^2 =
    - 1$ where $\ring{R}$ is ordered, such a bilinear form makes
    $\mathcal{E}$ a strong Morita equivalence bimodule with respect to
    $\mathcal{A}$ and $\End_{\mathcal{A}}(\mathcal{E})$, see e.g.
    \cite{bursztyn.waldmann:2005b}.
\end{remark}
\begin{definition}[Courant Algebroid]
    \label{definition:CourantAlgebroidStructure}
    A Courant algebroid structure on $\Eps$ is a $\ring{R}$-bilinear map
    \begin{equation}
        [\cdot,\cdot]_\mathcal{C}:\Eps \otimes \Eps \longrightarrow \Eps,
    \end{equation}
    called Courant bracket, together with a map, the anchor,
    \begin{equation}
        \sigma:\Eps \longrightarrow \Der(\A)
    \end{equation}
    such that for all $x,y,z \in \Eps$ and $a\in \A$ we have:
    \begin{enumerate}
    \item $(\Eps,[\cdot,\cdot]_\mathcal{C})$ is a Leibniz algebra,
        i.e. $[\cdot,\cdot]_\mathcal{C}$ satisfies the Jacobi identity
        \begin{equation}
            \label{eq:def-courant-jacobi}
            [x,[y,z]_\mathcal{C}]_\mathcal{C} =
            [[x,y]_\mathcal{C},z]_\mathcal{C} +
            [y,[x,z]_\mathcal{C}]_\mathcal{C},
        \end{equation}
    \item
        \begin{equation}\label{eq:def-courant-derivation}
            \sigma(x)\SP{y,z} = \SP{[x,y]_\mathcal{C},z} +
            \SP{y,[x,z]_\mathcal{C}}, 
        \end{equation}
    \item
        \begin{equation}\label{eq:def-courant-antisymm}
            \sigma(x)\SP{y,z} = \SP{x,[y,z]_\mathcal{C}+[z,y]_\mathcal{C}}.
        \end{equation}
    \end{enumerate}
\end{definition}

Due to fullness we can write every $a\in \A$ as a finite sum $a =
\sum_i\SP{ x_i,y_i}$ with $x_i,y_i \in \Eps$. Using
Equation~(\ref{eq:def-courant-antisymm}) we get then for all $z \in
\Eps$ that
\begin{equation}
    \sigma(z) a =\sum_i \sigma(z)\SP{x_i,y_i} =
    \sum_i\SP{z,[x_i,y_i]_\mathcal{C} + [y_i,x_i]_\mathcal{C}}, 
\end{equation}
hence the map $\sigma$ for a given Courant algebroid structure is
uniquely determined. Moreover, it follows that $\sigma$ is an
$\A$-module homomorphism.
\begin{example}
    Consider the space of sections $\Gamma^\infty(E)$ of a smooth
    (finite dimensional) vector bundle $E\longrightarrow M$ over a
    manifold $M$ with non-zero fiber dimension.  By the
    Serre-Swan-Theorem, $\Gamma^\infty(E)$ is a projective, finitely
    generated module over $C^\infty(M)$, and every non-degenerate
    fibre metric is known to induce a strongly non-degenerate and full
    $C^\infty(M)$-valued inner product on $\Gamma^\infty(E)$.  The
    definition of a Courant algebroid structure on $\Gamma^\infty(E)$
    coincides then with the usual (non-skew-symmetric) definition of a
    Courant algebroid given in \cite{roytenberg:1999a}.
\end{example}

One can show that $[\cdot,\cdot]_{\mathcal{C}}$ satisfies a Leibniz
rule in the second argument, i.e. that
\begin{equation}
    [x,ay]_\mathcal{C} = a [x,y]_\mathcal{C} + (\sigma(x)a) y
\end{equation}
for all $x,y \in \Eps$ and $a \in \A$. The proof is the same as in the
vector bundle case, see \cite{uchino:2002a}. Using the faithfulness of
our module one can further show that $\sigma:\Eps\longrightarrow
\Der(\mathcal{A})$ is a homomorphism of Leibniz algebras, i.e. an
action of the Leibniz algebra $\Eps$ on $\A$ via derivations.

Definition~\ref{definition:CourantAlgebroidStructure} generalizes the
concept of a Courant algebroid structure on a vector bundle like
Lie-Rinehart pairs generalize Lie algebroids, see
\cite{rinehart:1963a, huebschmann:1990a}. In fact, if $\mathcal{L}
\subset \Eps$ is a coisotropic submodule which is closed under the
Courant bracket, then $(\mathcal{A},\mathcal{L})$ is a Lie-Rinehart
pair.

Since we require the bilinear form $\SP{\cdot,\cdot}$ to be
non-degenerate, a Courant bracket is skew-symmetric if and only if
$\sigma(x) = 0$ for all $x \in \Eps$. In this case $\Eps$ is a Lie
algebra over $\A$.
\begin{example}\label{example:semisimpleLie}
   Let $\mathfrak{g}$ be a finite dimensional, semi-simple Lie algebra
   over a field $\mathbb{k}$ of characteristic zero. Then the Killing form on
   $\mathfrak{g}$ is an $\ad$-invariant, non-degenerate symmetric
   bilinear form. Hence $\Eps = \mathfrak{g}$ with Killing form and trivial
   anchor is a Courant algebroid over $\A = \mathbb{k}$.
\end{example}

\begin{remark}
    \label{remark:CatCourant}%
    It is rather obvious how to define morphisms of Courant algebroids
    as structure preserving maps. Then it is easy to see that Courant
    algebroids form a category.
\end{remark}

%
% The Poisson Algebra $\mathcal{C}^\bullet$
%

\section{The Poisson Algebra $\mathcal{C}^\bullet(\mathcal{E})$}
\label{ThePoissonAlgebraC}

In this section we construct a graded Poisson algebra of degree $-2$
containing all candidates for Courant algebroids structures on the
given module $\mathcal{E}$ for a fixed choice of the inner product.
\begin{definition}\label{def:quasiCourant}
    Let $r\geq 2$. The subset $\MC^r(\Eps) \subseteq
    \Hom_\ring{R}(\bigotimes_\ring{R}^{r-1} \Eps,\Eps)$  consists of all
    elements $\C$ for which there
    exists an $\ring{R}$-multilinear map
    \begin{equation}
        \label{eq:TheSymbol}
        \sigma_{\C} \in
        \Hom_\ring{R}(\textstyle\bigotimes_\ring{R}^{r-2} \Eps,
        \Der(\A)),
    \end{equation}
    called the symbol of $\C$, fulfilling the following conditions:
    \begin{enumerate}
    \item  For all $x_1,\ldots,x_{r-2},u,w\in \Eps$ we have
        \begin{equation}\label{multi.courant.der}
            \sigma_{\C}(x_1,\ldots,x_{r-2}) \SP{u,w} = 
            \SP{\C(x_1,\ldots,x_{r-2},u),w}
            +\SP{u,\C(x_1,\ldots,x_{r-2},w)}.
        \end{equation}
      \item For $r\geq 3$ and for all $x_1,\ldots,x_{r-1},u\in \Eps$
        and $1 \leq i \leq r-2$ we have that
        \begin{eqnarray}\label{multi.courant.antisymm} \notag
            &\SP{\C(x_1,\ldots,x_i,x_{i+1},\ldots, x_{r-1})
            +\C(x_1,\ldots,x_{i+1},x_i,\ldots, x_{r-1}),u} \\
            &\qquad =\sigma_{\C}(x_1,\elidetwo{i}{i+1},x_{r-1},u)
            \SP{x_i,x_{i+1}}.
        \end{eqnarray}
    \end{enumerate}
    Furthermore, we define $\MC^{0}(\Eps) = \A$ and $\MC^1(\Eps) =
    \Eps$ and set
    \begin{equation}    
        \label{eq:DerCooleComplex}
        \MC^\bullet(\Eps) = \bigoplus_{r\geq 0} \MC^r(\Eps).
    \end{equation}
\end{definition}

It is obvious from the definition that $\MC^\bullet(\Eps)$ is a graded
$\A$-module. Moreover, an element $m \in \MC^3(\Eps)$ is a  Courant
algebroid structure on $\Eps$ if and only if $m$ satisfies the Jacobi
identity, i.e. if and only if
\begin{equation}
    \label{eq:CourantBracketCondition}
    m(x,m(y,z)) = m(m(x,y),z) + m(y,m(x,z))
\end{equation}
for all $x,y,z \in \Eps$.
\begin{remark}
    In analogy to the terminology for higher Lie brackets (see
    e.g.~\cite{filippov:1985a, vinogradov.vinogradov:1998a}) elements
    in $\MC^r(\Eps)$ could be called quasi-$r$-Courant brackets. An
    $r$-Courant bracket would then be an element $\C \in \MC^r(\Eps)$
    which satisfies
    \begin{equation}
        \C(x_1,\ldots,x_{r-2},\C(y_1,\ldots,y_{r-1})) =
        \sum_{i=1}^{r-1}
        \C(y_1,\ldots,y_{i-1},
        \C(x_1,\ldots,x_{r-2},y_i),y_{i+1},\ldots,y_{r-1})
    \end{equation}
    for all $x_1,\ldots,x_{r-2},y_1,\ldots,y_{r-1} \in \Eps$.
\end{remark}
\begin{corollary}
    The symbol $\sigma_{\C}$ of $\C \in \MC^r(\Eps)$ is uniquely
    determined by $\C$.
\end{corollary}
\begin{proof}
    This is an easy consequence of the fullness of $\SP{\cdot,
      \cdot}$ and Equation~(\ref{multi.courant.der}).
\end{proof}
For $\C \in \MC^r(\Eps)$ with $r \geq 3$ and $a = \sum_i \SP{u_i,v_i}$
 condition \eqref{multi.courant.antisymm} implies that
\begin{equation}
    \begin{split}
        \sigma_{\C}(x_1,\ldots,x_{r-2}) a
        &=
        \sum_i \sigma_{\C}(x_1,\ldots,x_{r-2}) \SP{u_i,v_i}\\
        &= \sum_i \SP{\C(x_1,\ldots,u_i,v_i,\ldots,x_{r-3}) +
          \C(x_1,\ldots,v_i,u_i,\ldots,x_{r-3}),x_{r-2}}.
    \end{split}
\end{equation}
Hence the map $\sigma_{\C}$ is $\A$-linear in the last argument. Thus,
we have a well-defined map
\begin{equation}
    \label{eq:DCDef}
    \dif_{\C}  
    \in \Hom_\ring{R}
    \left(
        \textstyle\bigotimes_\ring{R}^{r-3} \Eps, 
        \Der(\A) \otimes_\A \Eps
    \right)
\end{equation}
given by
\begin{equation}
    \SP{\dif_{\C}(x_1,\ldots,x_{r-3})a,y} =
    \sigma_{\C}(x_1,\ldots,x_{r-3},y) a.
\end{equation} 
Note that we use here the isomorphism $\Der(\A,\Eps)\cong
\Der(\A)\otimes_\A \Eps$, which holds since $\Eps$ is supposed to be
projective and finitely generated.
\begin{lemma}
    \label{lemma.courantmaps-properties}
    Let $\C \in \MC^{r}(\Eps)$ with $r \geq 2$. Then
    \begin{equation}
        \label{derivation-rule}
        \C(x_1,\ldots,a x_{r-1}) = a \C(x_1,\ldots,x_{r-1})+
        \sigma_{\C}(x_1,\ldots,x_{r-2})a\; x_{r-1},
    \end{equation}   
    for all $x_1,\ldots,x_{r-1} \in \Eps$ and $a\in \A$.
\end{lemma}
\begin{proof}
    The proof can be done analogously to \cite{uchino:2002a}.
\end{proof}

Sometimes it is more convenient not to work with elements $\C \in
\MC^r(\Eps)$ but with $\ring{R}$-multilinear forms $\omega \in
\Hom_\ring{R}(\bigotimes_{\ring{R}}^r\Eps,\A)$ defined by
$\omega(x_1,\ldots,x_r) = \SP{\C(x_1,\ldots,x_{r-1}),x_r}$. This
motivates the following definition:
\begin{definition}
    For $r\geq 1$ the subspace $\Omega_\MC^r(\Eps) \subseteq
    \Hom_\ring{R}(\bigotimes_{\ring{R}}^r\Eps,\A) $
    consists of all elements $\omega$ fulfilling the following
    conditions:
    \begin{enumerate}
    \item $\omega$ is an $\A$-module homomorphism in the last entry, i.e.
        \begin{equation}
            \omega(x_1,\ldots,x_{r-1},a x_r) = a
            \omega(x_1,\ldots,x_{r-1},x_r)
        \end{equation}
        for all $a \in \A$.
      \item For $r \geq 2$ there exists an $\ring{R}$-multilinear map
        $\sigma_{\omega} \in
        \Hom_\ring{R}(\textstyle\bigotimes_\ring{R}^{r-2} \Eps,
        \Der(\A))$ such that
        \begin{equation}
            \begin{split}
                \omega(x_1,&\ldots,x_i,x_{i+1},\ldots, x_r)
                +\omega(x_1,\ldots,x_{i+1},x_i,\ldots, x_r) \\
                &\qquad =\sigma_{\omega}(x_1,\elidetwo{i}{i+1},x_r)
                \SP{x_i,x_{i+1}}
            \end{split}
        \end{equation}
        for all $1\leq i \leq r-1$.
    \end{enumerate}
    Furthermore, we set $\Comega^0(\Eps) = \A$ and
    $\Comega^\bullet(\Eps) = \bigoplus_{r = 0}^\infty
    \Comega^r(\Eps)$.
\end{definition}
The next lemma is clear.
\begin{lemma}
    \label{lemma.iso-maps-forms}
    There is an isomorphism of graded $\mathcal{A}$-modules
    \begin{equation}
        \label{eq:OmegaDef}
        \MC^\bullet(\Eps) \longrightarrow
        \Omega_\MC^\bullet(\Eps)
    \end{equation}
    given by
    \begin{equation}
        \label{eq:omegaCDef}
        \omega_C(x_1,\ldots,x_r) = \SP{\C(x_1,\ldots,x_{r-1}),x_r}
    \end{equation}
    for $r \ge 1$ and by the identity on $\mathcal{A} = \C^0(\Eps) =
    \Omega_\MC^0(\Eps)$.
\end{lemma}
\begin{definition}
    For an $\A$-module $\mathcal{F}$ we define
    \begin{equation}
        \SDer^p(\A,\mathcal{F}) =
        \{D\in\Hom_{\ring{R}}(\textstyle\bigotimes_{\ring{R}}^p
        \A,\mathcal{F})~|~ D 
        \text{ is symmetric and a derivation in each entry}\}
    \end{equation}
    and set $\SDer^p(\A) = \SDer^p(\A,\A)$.
\end{definition}

\begin{proposition}
    Let $\omega \in \Comega^r(\Eps)$ with $r\geq 2$. Then there exist
    unique maps
    \begin{equation}
        \pi^{(p)}_\omega \in \SDer^p(\A,\Comega^{r-2p}(\Eps))
    \end{equation}
    for $0 \leq 2p \leq r$ such that $\pi^{(0)}_\omega = \omega$ and
    such that $\pi^{(p+1)}_\omega (a_1,\ldots,a_p,\cdot)$ is the
    symbol of $\pi^{(p)}_\omega (a_1,\ldots,a_p)$ for all
    $a_1,\ldots,a_p \in \A$.
\end{proposition}
\begin{proof}
    We prove the proposition by induction over $p$. By definition
    $\pi^{(1)}_\omega = \sigma_\omega$, hence $\pi^{(1)}_\omega$ is a
    derivation. Denote by $i_x \omega$ the map which is obtained by
    inserting $x \in \Eps$ in the first argument of $\omega$. Then
    clearly $i_x \omega \in \Comega^{r-1}(\Eps)$. Since for $u,v \in
    \Eps$ one has    
    \[
        \pi^{(1)}_\omega (\SP{u,v}) = (i_u i_v + i_v
        i_u) \omega,
    \]
    it follows that $\pi^{(1)}_\omega \in \Der(\A,\Comega^{r-2}(\Eps))$ by
    fullness of the scalar product.  Suppose now that we have already found
    $\pi^{(0)}_\omega,\ldots,\pi^{(p)}_\omega$. If $2p \geq r-1$ there
    is nothing more to show. Otherwise define
    $\pi^{(p+1)}_\omega(a_1,\ldots,a_p,\cdot)$ as the symbol of
    $\pi^{(p)}_\omega (a_1,\ldots,a_p) \in \Comega^{r-2p}(\Eps)$ for
    all $a_1,\ldots,a_p \in \A$. Then
    \begin{align*}
        &\pi_\omega^{(p+1)}(a_1,\ldots,a_{p},\SP{u,v})
        (x_1,\ldots,x_{r-2p-2}) \\ 
        &= \pi_\omega^{(p)}(a_1,\ldots,a_p)
        (x_1,\ldots,u,v,\ldots,x_{r-2p-2})
        + \pi_\omega^{(p)}(a_1,\ldots,a_{p})
        (x_1,\ldots,v,u,\ldots,x_{r-2p-2}),
    \end{align*}
    and by fullness we conclude that $\pi^{(p+1)}_\omega \in
    \Hom(\bigotimes_\ring{R}^{p+1} \A,\Comega^{r-2(p+1)}(\Eps))$ and that
    $\pi^{(p+1)}_\omega$ is derivative and symmetric in the first $p$
    arguments. To show the symmetry in the last two arguments let
    $u,v,w,z \in \Eps$. Then
    \begin{align*}
        \pi_\omega^{(p+1)}(a_1,\ldots,&a_{p-1},\SP{u,v},\SP{w,z})
        (x_1,\ldots,x_{r-2p-2})\\ &=
        \pi_\omega^{(p)}(a_1,\ldots,a_{p-1},\SP{u,v})
        (x_1,\ldots,w,z,\ldots,x_{r-2p-2})\\
        &\quad + \pi_\omega^{(p)}(a_1,\ldots,a_{p-1},\SP{u,v})
        (x_1,\ldots,z,w,\ldots,x_{r-2p-2})\\[1mm]
        & =
        \pi_\omega^{(p-1)}(a_1,\ldots,a_{p-1})
        (x_1,\ldots,u,v,\ldots,w,z,\ldots,x_{r-2p-2})\\  
        &\quad+\pi_\omega^{(p-1)}(a_1,\ldots,a_{p-1})
        (x_1\ldots,u,v\ldots,z,w,\ldots,x_{r-2p-2})\\
        &\quad+ \pi_\omega^{(p-1)}(a_1,\ldots,a_{p-1})
        (x_1\ldots,v,u,\ldots,w,z,\ldots,x_{r-2p-2}) \\
        &\quad+
        \pi_\omega^{(p-1)}(a_1,\ldots,a_{p-1})
        (x_1\ldots,v,u,\ldots,z,w,\ldots,x_{r-2p-2})  
        \\[1mm]
        &= \pi_\omega^{(p)}(a_1,\ldots,a_{p-1},\SP{w,z})
        (x_1,\ldots,u,v,\ldots,x_{r-2p-2})\\
        &\quad + \pi_\omega^{(p)}(a_1,\ldots,a_{p-1},\SP{w,z})
        (x_1,\ldots,v,u,\ldots,x_{r-2p-2})\\[1mm]
        &=\pi_\omega^{(p+1)}(a_1,\ldots,a_{p-1},\SP{w,z},\SP{u,v})
        (x_1,\ldots,x_{n-2p-2}),
    \end{align*}
    and fullness implies now that $\pi^{(p+1)}_\omega \in
    \SDer^{p+1}(\A,\Comega^{r-2(p+1)}(\Eps))$.
\end{proof}
\begin{definition}
    Let $\C \in \MC^r(\Eps)$ with $r \geq 3$. Then define  maps
    \begin{equation}
        \delta_{\C}^{(p)}  \in \SDer^p(\A,\MC^{r-2p}(\Eps))
    \end{equation}
    for $ 0 \leq 2p < r$ by
    \begin{equation}
        \SP{\delta_{\C}^{(p)}(a_1,\ldots,a_p)(x_1,\ldots,x_{r-2p-1}),x_{r-2p}}
        = \pi_{\omega_{\C}}^{(p)}(a_1,\ldots,a_p)(x_1,\ldots,x_{r-2p-1},x_{r-2p}),
    \end{equation}
    where $a_1,\ldots,a_p \in \A$ and $x_1,\ldots,x_{r-2p} \in \Eps$,
    and $\omega_{\C} \in \Comega^{r}(\Eps)$ is given by
    \eqref{eq:omegaCDef}.
\end{definition}
Note that we have $\delta^{(0)}_{\C} = \C$ and $\delta^{(1)}_{\C} =
\dif_{\C}$.
\begin{example}
   Consider $\mathcal{A}$ as $\mathcal{A}$-module in the usual way and
   use the canonical inner product $\SP{a, b} = ab$.  Then it is an
   easy exercise to show that for any $P \in \SDer^2(\mathcal{A})$ the
   map
   \begin{equation}
      \label{eq:BadC}
      \C(x, y, z) = P(x, z) y
   \end{equation}
   satisfies all requirements from Definition~\ref{def:quasiCourant},
   and hence provides an element $\mathsf{C} \in \MC^4(\mathcal{A})$.
   Explicitly, we have $\pi_{\C}^{(1)}(a)(x,y) = \sigma_{\C} (x, y) a =
   P(x, a)y$ and $\delta_{\C}^{(1)}(a)(x) = \dif_{\C}(x) a = P(x, a)$, and
   finally $\pi_\C^{(2)}(a,b) = P(a,b)$.
\end{example}
Note that in the previous example $\pi_\C^{(2)}$ is in general an
element of $\SDer^2(\A)$, and not of the smaller space
$\Sym_\A^2\Der(\A)$ which we will introduce in
Definition~\ref{def:symmetric_product_of_Der}. 
In Section~\ref{sec:TheIsomorphism}, such symmetric biderivations will
show unpleasant features unless they factorize into symmetric 
products of derivations. The main purpose of the construction of the
Rothstein algebra in Section~\ref{TheRothsteinAlgebraR} will be to
show that they do not contribute to the deformation problem.

For $\C \in \MC^r(\mathcal{E})$ with $r\geq 2$ we denote
by $i_x \C$ the map with one argument less which is obtained by
inserting $x$ in the \emph{first} argument of $\C$.  It follows then
immediately from the definition of $\MC^r(\Eps)$ that $i_x \C \in
\MC^{r-1}(\Eps)$.
\begin{lemma}
    \label{lemma:ixC}
    Let $\C \in \MC^r(\Eps)$, $x \in \mathcal{E}$, and $a\in \A$. Then
    for $r = 3$ we have
    \begin{equation}
        \sigma_{i_x \C}a = i_x\sigma_{\C}a = \SP{\dif_{\C}a,x},
    \end{equation}
    while for $r \geq 4$ we get
    \begin{equation}
        \sigma_{i_x \C}a = i_x\sigma_{\C}a
        \quad
        \textrm{ as well as }
        \quad
        i_x \dif_{\C}a =  \dif_{i_x \C}a.
    \end{equation}
\end{lemma}
\begin{proof}
  This follows directly from the definition of $\MC^\bullet(\Eps)$ by
  some easy calculations.
\end{proof}
\begin{proposition}
    \label{proposition:CRN-Bracket}
    \begin{enumerate}
    \item There exists a unique $\ring{R}$-bilinear, graded
        skew-symmetric map
        \begin{equation}
            \MC^r(\Eps)\times \MC^s(\Eps) \longrightarrow \MC^{r+s-2}(\Eps),
        \end{equation}
        uniquely defined by 
        \begin{align}
            [a,b] & = 0,\label{eq:def.CRN-Bra.00} \\
            [a,x] & = 0\label{eq:def.CRN-Bra.01} = [x,a],\\
            [x,y] & = \SP{x,y},\label{eq:def.CRN-Bra.11}\\
            [\mathsf{D},a] & = \sigma_{\mathsf{D}} a =
            -[a,\mathsf{D}], \label{eq:def.CRN-Bra.20}\\ 
            [\C,x] & = i_x \C = (-1)^{r+1} [x,\C], \label{eq:def.CRN-Bra.r1}
        \end{align}
        for elements $a,b \in \A$, $x,y \in \Eps$, $\mathsf{D}\in \MC^2(\Eps)$
        and $\C \in \MC^r(\Eps)$ for $r\geq 2$, and by the recursion
        \begin{equation}\label{eq:CRN_recursion}
            i_x[\C_1,\C_2] = [[\C_1,\C_2],x] = (-1)^s [[\C_1,x],\C_2] +
            [\C_1,[\C_2,x]] 
        \end{equation}
        for $\C_1 \in \MC^r(\Eps)$, $\C_2 \in \MC^s(\Eps)$ and
        $x\in \Eps$.
    \item The bracket $[\cdot,\cdot]$ satisfies the graded Jacobi
        identity, i.e. for all $\C_1 \in \MC^r(\Eps)$, $\C_2 \in \MC^s(\Eps)$
        and $\C_3 \in \MC^t(\Eps)$ we have
        \begin{equation}
            [\C_1,[\C_2,\C_3]] = [[\C_1,\C_2],\C_3] + (-1)^{rs}
            [\C_2,[\C_1,\C_3]]. 
        \end{equation}
    \end{enumerate}
\end{proposition}
\begin{proof}
    We first show that the recursion~(\ref{eq:CRN_recursion}) is
    consistent with the above definitions for the low degrees. Due to
    the grading, the only thing to check is the consistency with
    (\ref{eq:def.CRN-Bra.r1}). But if $\mathsf{D} \in
    \MC^2(\Eps)$ and $x,y\in \Eps$, then
    \begin{equation}
        \label{eq:Jacobi.2.1.1}
        [[\mathsf{D},x],y] = \SP{\mathsf{D}(x),y} 
        = -\SP{\mathsf{D}(y),x} + \sigma_{\C}\SP{x,y}
        = -[[\mathsf{D},y],x] + [\mathsf{D},[x,y]].
    \end{equation}
    If $\C\in \MC^r(\Eps)$ with $r\geq 3$ then
    \[
    [[\C,x],y] = i_y i_x \C
    = -i_x i_y \C + \dif_{\C}\SP{x,y}
    = -[[\C,y],x] + \dif_{\C}\SP{x,y},
    \]
    hence (\ref{eq:CRN_recursion}) and fullness imply that $[\C,a] =
    \dif_\C a$. 
    This is still consistent since Lemma~\ref{lemma:ixC} implies
    that 
    \begin{equation}
        \label{eq:Jacobi.>=3.0.1}
        [[\C,a],x]  = [\dif_{\C} a,x] 
        = \left\{\begin{array}{rcl}
             \SP{x,\dif_{\C}a}  & =\sigma_{i_x \C}a & \text{if $r = 3$}\\
            i_x \dif_{\C} a & =\dif_{i_x \C}a   & \text{if $r \geq 4$}
        \end{array}\right\}
        = [[\C,x],a] + [\C,[a,x]].
    \end{equation}
    By construction it is clear that for $\C_1\in \MC^r(\Eps)$ and
    $\C_2\in \MC^s(\Eps)$ the recursively defined map $[\C_1,\C_2]$ is
    an element in $\Hom_\ring{R}(\bigotimes_\ring{R}^{r+s-3}\Eps,\Eps)$. It is
    further obvious that the bracket is graded skew-symmetric.  What
    remains to show is that $[\C_1,\C_2]$ is an element in
    $\MC^{r+s-2}(\Eps)$. We will prove this, together with the formula
    \[
    [[\C_1,\C_2],a] = [[\C_1,a],\C_2] + [\C_1,[\C_2,a]],
    \]
    by induction over $N = r+s$.  For $N = 1,2,3$ there is nothing
    more to show. Consider the case $N = 4$. If $a \in \A$ and $\C \in
    \MC^4(\Eps)$, then $[\C,a] = \dif_{\C} a \in \MC^2(\Eps)$, and
    moreover
    \begin{equation}
        \begin{split}
           \label{eq:Jacobi.4.0.0}
            [[\C,a],b] &= [\dif_{\C} a,b] = \delta_{\C}^{(2)}(a,b) =
            \delta_{\C}^{(2)}(b,a)\\
            & =  [\dif_{\C} b,a] = [[\C,b],a] = [[\C,b],a] +
            [\C,[a,b]].
        \end{split}
    \end{equation}
    For $x\in \Eps$ and $\C \in \MC^3(\Eps)$ we have $[\C,x] = i_x \C
    \in \MC^2(\Eps)$ and further
    \[
        [[\C,x],a] = [[\C,a],x] + [\C,[x,a]]
    \]
    by Equation~(\ref{eq:Jacobi.>=3.0.1}).  By a direct calculation we
    find further that for $\mathsf{D}_1,\mathsf{D}_2\in \MC^2(\Eps)$
    the bracket $[\mathsf{D}_1,\mathsf{D}_2]$ is again in
    $\MC^2(\Eps)$ and that
    \[
    \sigma_{[\mathsf{D}_1,\mathsf{D}_2]}a = \sigma_{\mathsf{D}_1}\sigma_{\mathsf{D}_2}a -
    \sigma_{\mathsf{D}_2}\sigma_{\mathsf{D}_1} a =
    [\sigma_{\mathsf{D}_1}a,\mathsf{D}_2] +
    [\mathsf{D}_1,\sigma_{\mathsf{D}_2}a],
    \]
    i.e.
    \begin{equation}
        \label{eq:Jacobi.2.2.0}
        [[\mathsf{D}_1,\mathsf{D}_2],a] = [[\mathsf{D}_1,a],\mathsf{D}_2]
        + [\mathsf{D}_1,[\mathsf{D}_2,a]]. 
    \end{equation} 
    So let $\C_1 \in \MC^r(\Eps)$ and $\C_2
    \in \MC^s(\Eps)$ with $r+s \geq 5$. By~(\ref{eq:CRN_recursion}) we have
    \[
    [[\C_1,\C_2],x] = (-1)^s [[\C_1,x],\C_2] + [\C_1,[\C_2,x]],
    \]
    and by induction we conclude that $ [[\C_1,\C_2],x]\in
    \MC^{r+s-3}(\Eps)$  for all $x \in \Eps$ and further that
    \begin{align*}
        [[[\C_1,\C_2],x],a] &= (-1)^s [[[\C_1,x],\C_2],a] +
        [[\C_1,[\C_2,x]],a]\\
        &= (-1)^s [[[\C_1,x],a],\C_2] + (-1)^s [[\C_1,x],[\C_2,a]] +
        [[\C_1,a],[\C_2,x]] + [\C_1,[[\C_2,x],a]]\\
        &= (-1)^s [[[\C_1,a],x],\C_2] + (-1)^s [[\C_1,x],[\C_2,a]] +
        [[\C_1,a],[\C_2,x]] + [\C_1,[[\C_2,a],x]]\\
        & = [[[\C_1,a],\C_2] + [\C_1,[\C_2,a]],x ]
    \end{align*}
    for all $a\in \A$.
    Consider now the map $h:\A \longrightarrow \MC^{r+s-4}(\Eps)$
    defined by
    \[
    h(a) =   [[\C_1,a],\C_2 ] + [\C_1,[\C_2, a]].
    \]
    We know that the map
    \[
    a \mapsto [[[\C_1,\C_2],x],a] = [h(a),x]
    \]
    is a derivation. Since obviously $[a \C,x] = a [\C,x]$ for all $a\in
    \A$, $x\in \Eps$ and $\C \in \MC^\bullet(\Eps)$ it follows that
    \[
    [h(ab),x] = [ah(b) + h(a) b,x]
    \]
    for all $a,b \in \A$ and $x\in \Eps$. But since the degree of $h(a)$
    is at least one, this implies that
    \[
    h(ab) = ah(b) + h(a) b
    \]
    for all $a,b \in \A$, i.e. $h \in \Der(\A,\MC^{r+s-4}(\Eps))$.  By
    construction we have $i_x h(a) = \dif_{[[\C_1,\C_2],x]}a$, hence
    \begin{align*}
        \SP{h(\SP{u,v})(x_1,\ldots, x_{r+s-3}),x_{r+s-2}}
        &= \SP{[\C_1,\C_2](x_1,\ldots,x_{r+s-2},u),v}\\
        &\quad+ \SP{[\C_1,\C_2](x_1,\ldots,x_{r+s-2},v),u}
    \end{align*}
    and 
    \begin{align*}
        h(\SP{x_i,x_{i+1}})(x_1,\elidetwo{i}{i+1},
        x_{r+s-1})
        &= [\C_1,\C_2](x_1,\ldots,x_i,x_{i+1},\ldots,x_{r+s-1})\\
        &\quad+ [\C_1,\C_2](x_1,\ldots,x_{i+1},x_i,\ldots,x_{r+s-1})
    \end{align*}
    for all $x_1,\ldots,x_{r+s-1},u,v \in \Eps$ and $2\leq i \leq
    r+s-2$.  It remains to show the last equation also for $i = 1$.
    But by the recursion rule we find that
    \begin{align*}
        [[[\C_1,\C_2],x],y] + [[[\C_1,\C_2],y],x] &= [[\C_1,[x,y]],\C_2] +
        [\C_1,[\C_2,[x,y]]]\\ 
        & = [[\C_1, \SP{x,y}],\C_2] + [\C_1,[\C_2,\SP{x,y}]]
    \end{align*}
    and hence
    \begin{align*}
        h(\SP{x_1,x_2})(x_3,\ldots, x_{r+s-1})
        &= [\C_1,\C_2](x_1,x_2,x_3,\ldots,x_{r+s-1})\\
        &\quad+ [\C_1,\C_2](x_2,x_1,x_3,\ldots,x_{r+s-1}).
    \end{align*}
    Therefore $[\C_1,\C_2] \in \MC^{r+s-2}(\Eps)$ with symbol
    \[
    \sigma_{[\C_1,\C_2]}(x_1,\ldots,x_{r+s-2})a =
    \SP{h(a)(x_1,\ldots,x_{r+s-3}),x_{r+s-2}}
    \]
    and
    \[
    \dif_{[\C_1,\C_2]} a = h(a) = [\dif_{\C_1}a,\C_2] +
    [\C_1,\dif_{\C_2} a ].
    \]
    Consider now the Jacobiator
    \[
    J(\C_1,\C_2,\C_3) = [\C_1,[\C_2,\C_3]] - [[\C_1,\C_2],\C_3] - (-1)^{rs}
    [\C_2,[\C_1,\C_3]]
    \]
    with $\C_1 \in \MC^r(\Eps)$, $\C_2 \in \MC^s(\Eps)$ and $\C_3 \in
    \MC^t(\Eps)$. We will show by induction over $N = r+s+t$ that the
    Jacobiator vanishes. Because of skew-symmetry we can assume that
    $r\leq s \leq t$. Since $[\cdot,\cdot]$ is of degree $-2$, there
    is nothing to prove for $N = 0,1,2,3$.  With the case $N = 4$ we
    are already done due to the identities~(\ref{eq:Jacobi.2.1.1}),
    (\ref{eq:Jacobi.>=3.0.1}), (\ref{eq:Jacobi.4.0.0}) and
    (\ref{eq:Jacobi.2.2.0}). Using (\ref{eq:CRN_recursion}) we further
    get
    \begin{align*}
        [J(\C_1,\C_2,\C_3),x] &=  [J(\C_1,\C_2,\C_3),x]\\
        & = (-1)^{s+t} J([\C_1,x],\C_2,\C_3) + (-1)^t
        J(\C_1,[\C_2,x],\C_3) +  J(\C_1,\C_2,[\C_3,x])
    \end{align*}
    for all $x \in \Eps$, and induction yields now $[J(\C_1,\C_2,\C_3),x]= 0$
    for all $x\in \Eps$. Since for $N\geq 5$ the Jacobiator
    $J(\C_1,\C_2,\C_3)$ is at least of degree one, we conclude that also
    $J(\C_1,\C_2,\C_3)=0$.
\end{proof}
\begin{corollary}\label{cor:Formula_Bracket}
    For $\C_1 \in \MC^r(\Eps)$, $\C_2 \in \MC^s(\Eps)$ with $r,s \geq
    2$ the bracket $[\C_1,\C_2]$ coincides with the bracket defined in
    \cite{balavoine:1995a, rotkiewicz:2005a} restricted to
    $\bigoplus_{r\geq 2}\MC^r(\Eps)$ after adapting the sign
    conventions appropriately.
\end{corollary}
The bracket defined in Proposition~\ref{proposition:CRN-Bracket} makes
$\MC^\bullet(\Eps)$ to a graded Lie algebra. The more explicit
formulas from \cite{balavoine:1995a, rotkiewicz:2005a} will not be
needed in the sequel.

In a next step we define an associative, graded commutative product
$\wedge$ on $\MC^\bullet(\Eps)$:
\begin{proposition}
    \label{proposition:WedgeProduct}
    There exists an associative, graded commutative $\ring{R}$-bilinear
    product $\wedge$ of degree zero on $\MC^\bullet(\Eps)$ uniquely defined by
    the equations
    \begin{equation}\label{eq:wedge.0.0}
        a \wedge b = a b = b \wedge a
    \end{equation}
    and
    \begin{equation}\label{eq:wedge.0.1}
        a \wedge x = a x = x \wedge a
    \end{equation}
    for all $a, b \in \A$ and $x \in \Eps$, and by the recursion rule
    \begin{equation}\label{eq:wedge-recursion-rule}
        [\C_1\wedge \C_2,x]
        = (-1)^s [\C_1,x]\wedge \C_2 + \C_1 \wedge [\C_2,x]
    \end{equation}
    for all $\C_1 \in \MC^r(\Eps)$, $\C_2 \in \MC^s(\Eps)$ and $x \in
    \Eps$.
\end{proposition}
\begin{proof}
   Clearly the recursion rule (\ref{eq:wedge-recursion-rule}) is
   consistent with the definitions (\ref{eq:wedge.0.0}) and
   (\ref{eq:wedge.0.1}).  Moreover, if $\wedge$ exists, it must be
   homogeneous of degree zero and graded commutative.  We prove now by
   induction over $N = r+s$ that for $\C_1 \in \MC^r(\Eps)$ and $\C_2 \in
   \MC^s(\Eps)$ the map
    \[
    (x_1,\ldots,x_{r+s-1}) \longmapsto (i_{x_1}(\C_1\wedge
    \C_2))(x_2,\ldots,x_{r+s-1}) 
    \]
    is an element in $\MC^{r+s}(\Eps)$, and that
    \[
    [\C_1\wedge \C_2,a] 
    =  [\C_1,a] \wedge \C_2 +  \C_1 \wedge [\C_2,a]
    \]
    for all $a\in \A$.  For $N = 0,1$ there is nothing to show.  If
    $a\in \A$ and $\mathsf{D} \in \MC^2(\Eps)$, then
    (\ref{eq:wedge-recursion-rule}) implies that
    \[
    [a\wedge \mathsf{D} ,x] 
    = a\wedge [\mathsf{D},x] 
    = a \mathsf{D}(x)
    = [a \mathsf{D},x] 
    \]
    for all $x\in\Eps$, hence $a \wedge \mathsf{D} = a \mathsf{D} \in
    \MC^2(\Eps)$ and $[a\wedge \mathsf{D},b] = \sigma_{a \mathsf{D}}b
    = a \sigma_{\mathsf{D}} b = a \wedge [D,b]$.  For $x,y,z\in \Eps$
    we get $i_z(x\wedge y) = -\SP{x,z} y + x \SP{y,z}$, whence $x
    \wedge y \in \MC^2(\Eps)$ with vanishing symbol and thus $[x\wedge
    y, a ] = 0 = [x,a]\wedge y + x \wedge [a,y]$.  Suppose now $N =
    r+s \geq 3$. By induction we find that
    \begin{align*}
        [i_x (\C_1,\wedge \C_2),a]& = [(-1)^s [\C_1,x]\wedge \C_2 + \C_1
        \wedge
        [\C_2,x],a]\\
        & = (-1)^s [[\C_1,x],a]\wedge \C_2 + (-1)^s [\C_1,x]\wedge
        [\C_2,a] +
        [\C_1,a]  \wedge  [\C_2,x] + \C_1 \wedge[[\C_2,x],a]\\
        & = (-1)^s [[\C_1,a],x]\wedge \C_2 + (-1)^s [\C_1,x]\wedge
        [\C_2,a] +
        [\C_1,a]  \wedge  [\C_2,x] + \C_1 \wedge[[\C_2,a],x]\\
        & = [[\C_1,a]\wedge \C_2 + \C_1 \wedge [\C_2,a],x].
    \end{align*}
    Hence the map $a \longmapsto [[\C_1,a]\wedge \C_2 + \C_1 \wedge
    [\C_2,a],x]$ is a derivation. Since the degree of $[\C_1,a]\wedge
    \C_2 + \C_1 \wedge [\C_2,a]$ is at least one, the map
    \[
    h(a) =  [\C_1,a]\wedge \C_2 + \C_1 \wedge
    [\C_2,a]
    \]
    is also a derivation. We have to show that $h$ is the map
    $\dif_{\C_1\wedge \C_2}$. By construction we already know that
    $[h(a),x] = \dif_{ i_x (\C_1\wedge \C_2)}a$. With a short
    calculation using the recursion rule we further find that
    \begin{align*}
        (i_y i_x + i_x i_y)(\C_1\wedge \C_2) & = (i_y i_x + i_x i_y) \C_1
        \wedge \C_2  + \C_1\wedge  (i_y i_x + i_x i_y)\C_2\\
        &= [\C_1,\SP{x,y}]
        \wedge \C_2  + \C_1\wedge  [\C_2,\SP{x,y}]\\
        &= h(\SP{x,y}),
    \end{align*}
    hence $\C_1\wedge \C_2$ is in fact in $\MC^{r+s}(\Eps)$ and
    $\dif_{\C_1\wedge \C_2} a = [\C_1,a]\wedge \C_2 + \C_1 \wedge
    [\C_2,a]$.  The associativity can now easily be proven by
    induction in a similar way we have proven the Jacobi identity for
    $[\cdot,\cdot]$ in Theorem~\ref{proposition:CRN-Bracket}.
\end{proof}   
\begin{corollary}
    Let $\C_1\in \MC^r(\Eps)$, $\C_2\in\MC^s(\Eps)$ with $r,s \geq 1$.
    Then $\C_1\wedge \C_2$ is given by
    \begin{equation*}\begin{split}
      \C_1\wedge \C_2 &(x_1,\ldots,x_{r+s-1}) \\
      &= (-1)^{rs} \sum_{\pi
        \in\shuffle{r,s-1}}\operatorname{\sign}(\pi)
      \SP{\C_1(x_{\pi(1)},\ldots,x_{\pi(r-1)}),x_{\pi(r)}}
      \C_2(x_{\pi(r+1)}\ldots,x_{\pi(r+s-1)}) \\
      &\quad + \sum_{\pi \in\shuffle{s,r-1}} \operatorname{\sign}(\pi)
      \SP{\C_2(x_{\pi(1)},\ldots,x_{\pi(s-1)}),x_{\pi(s)}}
      \C_1(x_{\pi(s+1)}\ldots,x_{\pi(r+s-1)}),
    \end{split}
  \end{equation*}
  where $\shuffle{p,q}$ denotes the $(p,q)$-shuffle permutations.
\end{corollary}
\begin{proof}
    The proof can be done by induction over $N = r+s$.
\end{proof}

Since $\MC^\bullet(\Eps) \cong \Comega^\bullet(\Eps)$, we can
transport the product $\wedge$ to $\Comega^\bullet(\Eps)$ with the
isomorphism given in Lemma~\ref{lemma.iso-maps-forms}, i.e. we set
\begin{equation*}
  \omega_{\C_1}\wedge \omega_{\C_2} = \omega_{\C_1\wedge
    \C_2}
\end{equation*}
for all $\C_1,\C_2 \in \MC^\bullet(\Eps)$. A little computation yields
then to the next corollary.
\begin{corollary}
  Let $\omega_1 \in \Comega^r(\Eps)$ and $\omega_2\in
  \Comega^s(\Eps)$ with  $r,s \geq 1$, then $\omega_1\wedge\omega_2$ is given
  by 
  \begin{equation*}
    \begin{split}
      \omega_1\wedge\omega_2&(x_1,\ldots,x_{r+s})\\
      &= (-1)^{rs}\sum_{\pi \in\shuffle{r,s}}\operatorname{\sign}(\pi)
      \omega_1(x_{\pi(1)},\ldots,x_{\pi(r)})
      \omega_2(x_{\pi(r+1)}\ldots,x_{\pi(r+s)}).
    \end{split}  \end{equation*}
  For $a\in \A$ and $\omega \in \Comega^\bullet(\Eps)$ we have
  \begin{equation*}
    a \wedge \omega =  a\omega = \omega \wedge a.
  \end{equation*}
\end{corollary}  
We can now formulate the main theorem of this section:
\begin{theorem}
    The triple $(\MC^\bullet(\Eps),[\cdot,\cdot],\wedge)$ is a graded Poisson
    algebra of degree $-2$.
\end{theorem} 
\begin{proof} 
    It remains to prove the Leibniz rule
    \begin{equation}\label{eq:courant-leibniz}
        [\C_1,\C_2\wedge\C_3] = [\C_1,\C_2]\wedge\C_3 +
        (-1)^{rs} \C_2\wedge[\C_1,\C_3]
    \end{equation}
    for $\C_1 \in \MC^r(\Eps)$, $\C_2 \in \MC^s(\Eps)$ and $\C_3 \in
    \MC^t(\Eps)$. We will do this by induction over $N = r+s+t$. For
    $N=0,1$ there is nothing to show. For $N=2$ we have the following
    identities, where $a,b\in \A$, $x,y\in \Eps$ and $\mathsf{D}\in
    \MC^2(\Eps)$:
    \begin{align*}
    [\mathsf{D},ab] 
    &= \sigma_{\mathsf{D}}(ab) 
    = \sigma_{\mathsf{D}}(a) b + a \sigma_{\mathsf{D}}(b) =
    [\mathsf{D},a]b + a [\mathsf{D},b]
    \\[2mm]
    [x,ay] &= \SP{x,ay} = a\SP{x,y} = a [x,y] + [x,a] y
    \\[2mm]
    [a,x\wedge y] &= 0 = [a,x] y + x [a,y]
    \\[2mm]
    [a,b D] &= - \sigma_{b\mathsf{D}} a = -b \sigma_{\mathsf{D}} a =
    [a,b] \mathsf{D} + b [a,\mathsf{D}]
    \end{align*}
    We can now finish the proof by induction using the
    Equations~(\ref{eq:CRN_recursion}) and
    (\ref{eq:wedge-recursion-rule}).
\end{proof}
\begin{remark}[Courant bracket as derived bracket]
    \label{remark:DerivedBracket}
    From \eqref{eq:CourantBracketCondition} one obtains that $m \in
    \MC^3(\Eps)$ defines a Courant algebroid structure on $\Eps$ if and only
    if $[m , m] = 0$. In this case, the Courant bracket corresponding to $m$
    is the derived bracket
    \begin{equation}
        \label{eq:CourantDerivedBracket}
        [x, y]_{m} = [[x,m],y]\quad \text{for all } x,y\in \Eps
    \end{equation}
    in the  sense of \cite{kosmann-schwarzbach:2004b}.
\end{remark}
\begin{remark}[Deformation theory, I]
    \label{remark:DeformationTheoryI}
    Let $m \in \MC^3(\Eps)$ be a Courant algebroid structure,
    i.e. $[m, m] = 0$, and hence $\delta_m = [m,\cdot]$ squares to
    zero. We hence get a cochain complex
    \begin{equation}
        \label{eq:TheComplex}
        \xymatrix{ \A \ar[r]^-{\delta_m} & \Eps
          \ar[r]^-{\delta_m} & \MC^2(\Eps) \ar[r]^-{\delta_m} &  \MC^3(\Eps)
          \ar[r]^-{\delta_m} &  \MC^4(\Eps)
          \ar[r]^-{\delta_m} &  }
        \cdots.
    \end{equation}
    Denote by $H^\bullet(\MC(\Eps),\delta_m)$ the cohomology of this
    complex. Since $\delta_m$ is an (even inner) derivation, the
    cohomology inherits the Poisson algebra structure of $\MC(\Eps)$.
    By the usual considerations one finds that
    $H^2(\MC(\Eps),\delta_m)$ are the outer derivations of $m$, that
    $H^3(\MC(\Eps),\delta_m)$ parametrizes the non-trivial
    infinitesimal deformations $m_t = m + tm_1 + \cdots$ of $m$ up to
    formal automorphisms, and that $H^4(\MC(\Eps),\delta_m)$ contains
    the obstructions for a recursive construction of formal
    deformations. Again, the construction enjoys good functorial
    properties. We do not spell out the rather obvious details here.
\end{remark}

%
% The Rothstein Algebra $\mathcal{R}^\bullet$
%

\section{The Rothstein Algebra $\mathcal{R}^\bullet(\mathcal{E})$}
\label{TheRothsteinAlgebraR}

In this section we shall now describe a completely different approach
to the complex $\mathcal{C}^\bullet(\Eps)$ by establishing a kind of
``symbol calculus'' for it. To this end, we have to choose an
additional structure, a connection, to construct the \emph{Rothstein
  algebra}.
\begin{definition}[Connection]
    \label{definition:Connection}
    A connection (or: covariant derivative) $\nabla$ for the module
    $\mathcal{E}$ is a map $\nabla: \Der(\mathcal{A}) \times
    \mathcal{E} \longrightarrow \mathcal{E}$ such that
    \begin{equation}
        \label{eq:NablaAlinearFirstArg}
        \nabla_{aD} x = a \nabla_D x
    \end{equation}
    \begin{equation}
        \label{eq:NablaLeibnizSecondArgument}
        \nabla_D(ax) = a \nabla_D x + D(a) x
    \end{equation}
    for all $a \in \mathcal{A}$, $D \in \Der(\mathcal{A})$, and $x \in
    \mathcal{E}$. If $\SP{\cdot, \cdot}: \mathcal{E} \times
    \mathcal{E} \longrightarrow \mathcal{A}$ is an
    $\mathcal{A}$-valued inner product, then $\nabla$ is called metric
    if in addition
    \begin{equation}
        \label{eq:NablaMetric}
        D \SP{x, y} = \SP{\nabla_D x, y} + \SP{x, \nabla_D y}
    \end{equation}
    for all $x, y \in \mathcal{E}$ and $D \in \Der(\mathcal{A})$.
\end{definition}
The following lemma is well-known and provides us a metric connection
for the module $\mathcal{E}$:
\begin{lemma}
    \label{lemma:LeviCivita}
    If $\mathcal{E}$ is finitely generated and projective then it
    allows for a connection $\nabla$. If $\mathcal{E}$ has in addition
    a strongly non-degenerate inner product $\SP{\cdot, \cdot}$, then
    $\nabla$ can be chosen to be a metric connection.
\end{lemma}
\begin{proof}
    If $\mathcal{E} = P\mathcal{A}^n$ with $P = P^2 \in
    M_n(\mathcal{A})$ then $\nabla_D Px = P D(x)$ is a connection
    where $D$ is applied componentwise to $x \in \mathcal{A}^n$.
    Moreover, if $\tilde{\nabla}$ is any connection and $\SP{\cdot,
      \cdot}$ is strongly non-degenerate then $\nabla$ defined by
    \[
    \SP{\nabla_D x, y}
    = \frac{1}{2}
    \left(\SP{\tilde{\nabla}_D x, y} 
        - \SP{x, \tilde{\nabla}_D y} + D\SP{x, y}
    \right)
    \]
    is easily shown to be a metric connection. Note that fullness is
    not needed here.
\end{proof}

Endow the algebra of symmetric multi-derivations
$\SDer^\bullet(\mathcal{A})$ with the obvious symmetric product $\vee$
given for $P \in \SDer^p(\mathcal{A})$ and $Q\in \SDer^q(\mathcal{A})$
by
\[
(P\vee Q)(a_1,\ldots,a_{p+q}) = \sum_{\pi \in \Sym_{p,q}}
P(a_{\pi(1)},\ldots,a_{\pi(p)})Q(a_{\pi(p+1)},\ldots,a_{\pi(p+q)})
\]
for all $a_1\ldots a_{p+q} \in \A$, where $\Sym_{p,q}$ denotes the set
of $(p,q)$-shuffles. This makes $\SDer^\bullet(\A)$ to an associative,
commutative algebra.
\begin{definition}\label{def:symmetric_product_of_Der}
   Denote by $\Sym_{\mathcal{A}}^\bullet \Der(\mathcal{A})$  the subalgebra of
   $\SDer^\bullet(\A)$  generated by $\A$ and $\Der(\A)$.
\end{definition}

Note that in general $\Sym_{\mathcal{A}}^\bullet \Der(\mathcal{A})$ 
is a proper subset of $\SDer^\bullet(\mathcal{A})$. Consider for
example the algebra $\A = \mathbb{R}[X]/(X^2)$, then $\Der(\A)\vee_\A
\Der(\A) = 0 \neq \SDer^2(\A)$.  In nice geometric contexts however, the
difference is absent:
\begin{example}
    \label{example:SymDerSDer}
    For a smooth manifold $M$ the symmetric multi-derivations
    $\SDer^k(C^\infty(M))$ of the smooth functions $\mathcal{A} =
    C^\infty(M)$ can be identified canonically with the smooth
    sections $\Gamma^\infty(\Sym^k TM)$ of symmetric powers of the
    tangent bundle. Moreover, by use of the Serre-Swan-Theorem one
    obtains that the $k$-th symmetric power of $\Der(C^\infty(M))
    \cong \Gamma^\infty(TM)$ is indeed in bijection to
    $\Gamma^\infty(S^k TM)$.
\end{example}
We can now define the Rothstein algebra as associative algebra as
follows. Note that as usual $\Sym_{\mathcal{A}}^0\Der(\mathcal{A}) =
\Anti_{\mathcal{A}}^0 \mathcal{E} = \mathcal{A}$ by convention.
\begin{definition}[Rothstein algebra]
    \label{definition:RothsteinAlgebra}
    The Rothstein algebra is defined by
    \begin{equation}
        \label{eq:RothsteinAlgebra}
        \Rothstein^\bullet(\mathcal{E}) 
        =
        \bigoplus_{r=0}^\infty \Rothstein^r (\mathcal{E})
        \quad
        \textrm{with}
        \quad
        \Rothstein^r(\mathcal{E}) = \bigoplus_{2p + k = r}
        \Sym_{\mathcal{A}}^p\Der(\mathcal{A}) \otimes_\mathcal{A}
        \Anti_\mathcal{A}^k \mathcal{E}, 
    \end{equation}
    where the tensor product is taken over $\mathcal{A}$, with the
    canonical product $\wedge$ defined on factorizing elements by
    \begin{equation}
        \label{eq:RothsteinProduct}
        (P \otimes \xi) \wedge (Q \otimes \eta)
        = (P \vee Q) \otimes (\xi \wedge \eta).
    \end{equation}
\end{definition}
With this definition, the following properties are immediate. Note
that the associative algebra structure of
$\Rothstein^\bullet(\mathcal{E})$ does not yet depend on the inner
product.
\begin{proposition}
    \label{proposition:RothsteinAlgebra}
    The Rothstein algebra $\Rothstein^\bullet(\mathcal{E})$ with the
    product \eqref{eq:RothsteinProduct} is an associative and graded
    commutative algebra with $\Rothstein^0(\mathcal{E}) = \mathcal{A}$
    as sub-algebra. Moreover, $\Rothstein^0(\mathcal{E})$,
    $\Rothstein^1(\mathcal{E})$  and
    $\Rothstein^2(\mathcal{E})$ generate
    $\Rothstein^\bullet(\mathcal{E})$.
\end{proposition}

Using a metric connection we can define a graded Poisson bracket of
degree $-2$ on the Rothstein algebra.  To this end we have to
introduce the curvature of $\nabla$. First it is clear that a given
connection $\nabla$ for $\mathcal{E}$ extends to
$\Anti^\bullet_{\mathcal{A}} \mathcal{E}$ by imposing the Leibniz rule
with respect to the $\wedge$-product. Thus we can consider
\begin{equation}
    \label{eq:CurvatureDef}
    R(D, E) \xi 
    = \nabla_D \nabla_E \xi
    - \nabla_E \nabla_D \xi
    - \nabla_{[D, E]} \xi,
\end{equation}
for $D, E \in \Der(\mathcal{A})$ and $\xi \in
\Anti^\bullet_{\mathcal{A}}(\mathcal{E})$. The usual computation shows
that $R(\cdot, \cdot) \cdot$ is $\mathcal{A}$-linear in all three
arguments. Thus it defines an element
\begin{equation}
    \label{eq:CurvatureTensor}
    R(D, E) \in \End_{\mathcal{A}} (\Anti^\bullet_{\mathcal{A}} \mathcal{E})
\end{equation}
in the $\mathcal{A}$-linear endomorphisms of
$\Anti^\bullet_{\mathcal{A}} \mathcal{E}$. Moreover, it clearly
preserves the anti-symmetric degree of $\Anti^\bullet_{\mathcal{A}}
\mathcal{E}$ whence it is homogeneous of degree $0$.  Finally, $R(D,
E)$ is a derivation of the $\wedge$-product as the commutator of
derivations is a derivation.  Restricting $R(D, E)$ to $\mathcal{E}$
gives a $\mathcal{A}$-linear map $R(D, E): \mathcal{E} \longrightarrow
\mathcal{E}$. Since $\nabla$ is metric, it follows that
\begin{equation}
    \label{eq:Rmetric}
    \SP{R(D, E)x, y} = - \SP{R(D, E)y, x},
\end{equation}
whence the map $(x, y) \mapsto \SP{R(D, E)x, y}$ is
$\mathcal{A}$-bilinear and anti-symmetric.  Using the strongly
non-degenerate inner product $\SP{\cdot, \cdot}$ on $\mathcal{E}$ this
allows to define $r(D, E) \in \Anti_{\mathcal{A}}^2 \mathcal{E}$ by
\begin{equation}
    \label{eq:rDef}
    \SP{R(D, E)x, y} = \SP{r(D, E), x \wedge y}.
\end{equation}
Directly from the definition of the curvature we obtain the Bianchi
identity
\begin{equation} \label{eq:BianchiForR}
  \begin{split}
    &[\nabla_{D_1},R(D_2,D_3)] + [\nabla_{D_2},R(D_3,D_1)]+
    [\nabla_{D_3},R(D_1,D_2)]\\
    &+ R(D_1,[D_2,D_3]) + R(D_2,[D_3,D_1]) +
    R(D_3,[D_1,D_2]) = 0
  \end{split} 
\end{equation}
for $D_1,D_2,D_3 \in \Der(\mathcal{A})$,
which reads for $r$ as
\begin{equation}
  \begin{split}\label{eq:BianchiForr}
    &\nabla_{D_1} r(D_2,D_3) + \nabla_{D_2} r(D_3,D_1)+
    \nabla_{D_3}r(D_1,D_2)\\
    &+ r(D_1,[D_2,D_3]) + r(D_2,[D_3,D_1]) +
    r(D_3,[D_1,D_2]) = 0.
  \end{split}
\end{equation}

With this preparation the Poisson structure can now be defined
analogously to the smooth case, see \cite{roytenberg:2002b}.
\begin{theorem}
    \label{theorem:RothsteinPoissonBracket}
    Let $\nabla$ be a metric connection on $\mathcal{E}$. Then there
    exists a unique graded Poisson structure $\RothBracket{\cdot,
      \cdot}$ on $\Rothstein^\bullet(\mathcal{E})$ of degree $-2$ such
    that
    \begin{align}
        \label{eq:RothsteinBracketOnGenerators}
        \RothBracket{a, b} &= 0 = \RothBracket{a, x},\\
        \RothBracket{x, y} &= \SP{x, y},\\
        \RothBracket{D,a}  &= -D(a),\\
        \RothBracket{D, x} &= - \nabla_D x, \textrm{ and}\\
        \RothBracket{D, E} &= - [D, E] - r(D, E),
    \end{align}
    for $a, b \in \mathcal{A} = \Rothstein^0(\mathcal{E})$, $x, y \in
    \mathcal{E} =\Rothstein^1(\mathcal{E})$, and $D, E \in
    \Der(\mathcal{A}) \subseteq \Rothstein^2(\mathcal{E})$.
\end{theorem}
\begin{proof}
    Since the Rothstein algebra is generated by the elements of degree
    $0$, $1$ and $2$, it will be sufficient to specify the Poisson
    bracket on these elements, which immediately gives uniqueness. The
    required graded version of the Leibniz rule is $\RothBracket{\phi,
      \psi \wedge \chi} = \RothBracket{\phi, \psi} \wedge \chi +
    (-1)^{rs} \psi \wedge \RothBracket{\phi, \chi}$ for $\phi \in
    \Rothstein^r(\mathcal{E})$ and $\psi \in
    \Rothstein^s(\mathcal{E})$. Clearly, this Leibniz rule is
    consistent with the definitions
    (\ref{eq:RothsteinBracketOnGenerators}) whence enforcing graded
    antisymmetry $\RothBracket{\phi, \psi} = - (-1)^{rs}
    \RothBracket{\psi, \phi}$ and the Leibniz rule extends
    $\RothBracket{\cdot, \cdot}$ to all of
    $\Rothstein^\bullet(\mathcal{E})$. It remains to show the graded
    Jacobi identity, which reads
    \begin{equation}
        \label{eq:JacobiForRothstein}
        \RothBracket{\phi, \RothBracket{\psi, \chi}}
        =
        \RothBracket{\RothBracket{\phi, \psi}, \chi}
        +
        (-1)^{rs}
        \RothBracket{\psi, \RothBracket{\phi, \chi}}
    \end{equation}
    for general elements. But clearly \eqref{eq:JacobiForRothstein} is
    fulfilled on generators thanks to metricity of the connection and
    the Bianchi identity for the curvature $r$.
\end{proof}
The Rothstein bracket depends therefore on the connection. However,
the next theorem says that this dependence is not crucial. To this
end, let $\nabla$ and $\nabla'$ be metric connections on $\Eps$.
Define for $D \in \Der(\mathcal{A})$ the module endomorphism $T_D \in
\End_\mathcal{A}(\Eps)$ by
\begin{equation}
    \label{eq:DifferenceConnection}
    T_D x = \nabla_D x - \nabla_D' x
\end{equation}
for all $x \in \Eps$. Then $\SP{T_D x,y} = -\SP{T_D y,x}$ for all
$x,y\in \Eps$, hence we get a well-defined $\mathcal{A}$-linear map
$t: \Der(\mathcal{A})\longrightarrow \Anti_{\mathcal{A}}^2\Eps$ by the
requirement $\SP{t(D),x\wedge y} = \SP{T_D x,y}$.  We extend this map
to the whole algebra $\Rothstein^\bullet(\Eps)$ by $t(a) = 0$, $t(x) =
0$ for $a\in \mathcal{A}$, $x\in \mathcal{E}$, and by enforcing the
Leibniz rule with respect to $\wedge$.
\begin{theorem}
    \label{theorem:ChangingTheConnection}
    Let $\nabla$ and $\nabla'$ be metric connections on $\Eps$, and
    let $\RothBracket{\cdot,\cdot}$ and $\RothBracket{\cdot,\cdot}'$
    be the associated Rothstein brackets. Let $t \in \End_{\ring{R}}
    (\Rothstein^\bullet(\Eps))$ be given as above. Then
    \begin{equation}
        \exp(t) = \sum_{n=0}^\infty \frac{t^n}{n!}:
        (\Rothstein^\bullet(\Eps),\RothBracket{\cdot,\cdot}) \longrightarrow
        (\Rothstein^\bullet(\Eps),\RothBracket{\cdot,\cdot}') 
    \end{equation}
    is an homogeneous isomorphism of degree zero of graded Poisson
    algebras.
\end{theorem}
\begin{proof}
    First of all notice that $\exp(t)$ is well-defined since $t$
    lowers the symmetric degree in
    $\Sym_{\mathcal{A}}^p\Der(\mathcal{A})$ by one.  By definition
    \[
    t: \Rothstein^\bullet(\Eps)\longrightarrow \Rothstein^\bullet(\Eps)
    \]
    is homogeneous of degree zero and a derivation of the
    $\wedge$-product. Hence $\exp(t)$ is also homogeneous of degree
    zero but now an automorphism of $\wedge$. To prove that $\exp(t)$
    maps $\RothBracket{\cdot, \cdot}$ to $\RothBracket{\cdot,
      \cdot}'$, it suffices to show this on generators, which follows
    by straightforward computations.
\end{proof}

The construction of the Rothstein algebra enjoys some nice functorial
properties: Let $\mathcal{F}$ be another finitely generated,
projective module over an algebra $\mathcal{B}$, together with a full,
strongly non-degenerate inner product $\SP{\cdot,\cdot}_\mathcal{F}$
and a metric connection whence we obtain a Rothstein algebra
$\Rothstein^\bullet(\mathcal{F})$.  Now, let be $g: \mathcal{A}
\longrightarrow \mathcal{B}$ an invertible algebra morphism and $G:
\mathcal{E} \longrightarrow \mathcal{F}$ a $\ring{R}$-linear isometric
module map along $g$, i.e. $G$ satisfies $G(a x) = g(a) G(x)$ and
\begin{equation}
    g(\SP{x,y}_\mathcal{E}) = \SP{G(x),G(y)}_\mathcal{F}
\end{equation}
for all $a\in \mathcal{A}$ and $x,y \in \mathcal{E}$. As before we
define a left inverse $H: \mathcal{F} \longrightarrow \mathcal{E}$ of
$G$ for $y\in \mathcal{F}$ by
\begin{equation}\label{eq:left_inverse_of_G}
    \SP{H(y),x}_\mathcal{E} =
    g^{-1}(\SP{y,G(x)}_\mathcal{F})\quad\text{for 
      all } x \in \mathcal{E}.
\end{equation}
Since $g$ is invertible we further have a map $g_* :\Der(\mathcal{A})
\longrightarrow \Der(\mathcal{B})$ given by $g_* D = g\circ D \circ
g^{-1}$.
\begin{proposition}
    \label{proposition:FunctorialRothstein}
    Let $G:\Eps\longrightarrow \mathcal{F}$ be an $\ring{R}$-linear
    isometric bijection along an algebra isomorphism $g:\A
    \longrightarrow \mathcal{B}$. Then $G$ lifts to a morphism
    \begin{equation}
        G_* :\Rothstein^\bullet(\mathcal{E}) \longrightarrow
        \Rothstein^\bullet(\mathcal{F})
    \end{equation}
    of Poisson algebras such that $G_*(a) = g(a)$ and $G_*(x) = G(x)$
    for all $a \in \mathcal{A}$ and $x \in \mathcal{E}$.
\end{proposition}
\begin{proof}
    Suppose first that we also have $G(\nabla^\mathcal{E}_D x) =
    \nabla^\mathcal{F}_{g_*D}(G(x))$ for all $D \in \Der(\mathcal{A})$
    and $x \in \mathcal{E}$. Then we define $G_*(a) = g(a)$, $G_*(x) =
    G(x)$ and $G_*(D) = g_*(D)$ for $a \in \mathcal{A}$, $x \in
    \mathcal{E}$ and $D \in \Der(\mathcal{A})$ and extend $G_*$ to the
    whole algebra $\Rothstein^\bullet(\mathcal{E})$ by enforcing it to
    be a algebra morphism with respect to the $\wedge$-products.  It
    follows that $G_*$ is a morphism of Poisson algebras as this is
    true on generators. For the general case define on $\mathcal{F}$
    another metric connection $\nabla'$ by $\nabla'_D y =
    G(\nabla^\mathcal{E}_{g^* D} H(y))$ for $D \in \Der(\mathcal{B})$
    and $y \in \mathcal{F}$, where $H$ is given by
    (\ref{eq:left_inverse_of_G}). Since we assume $G$ to be bijective
    we have $G^{-1} = H$, and one easily shows that $\nabla'$ is in
    fact a well defined metric connection for $\mathcal{F}$. Moreover
    $G(\nabla^\mathcal{E}_D x) = \nabla'_{g_*D}(G(x))$ for all $D \in
    \Der(\mathcal{A})$ and we hence get a Poisson morphism
    $\Rothstein^\bullet(\mathcal{E}) \longrightarrow
    \Rothstein'^\bullet(\mathcal{F})$, where
    $\Rothstein'^\bullet(\mathcal{F})$ denotes the Rothstein algebra
    together with the graded Poisson bracket constructed using
    $\nabla'$. Since by Theorem~\ref{theorem:ChangingTheConnection} we
    have a canonical isomorphism $ \Rothstein'^\bullet(\mathcal{F})
    \longrightarrow \Rothstein^\bullet(\mathcal{F})$ of graded Poisson
    algebras, the proof is complete.
\end{proof}

%
% The Isomorphism 
%

\section{The Symbol Calculus for $\mathcal{C}^\bullet(\mathcal{E})$}
\label{sec:TheIsomorphism}

In this section we find the relation between the two Poisson algebras
$\mathcal{C}^\bullet(\mathcal{E})$ and
$\Rothstein^\bullet(\mathcal{E})$. In particular, we will simplify the
cohomology for the deformation theory of Courant algebroid structures
from Remark~\ref{remark:DeformationTheoryI}.
\begin{definition}
    Let the $\mathcal{A}$-linear map $\mathcal{J} :
    \Rothstein^\bullet(\Eps) \longrightarrow \MC^\bullet(\Eps)$ be defined
    on generators by
    \begin{equation}
        \mathcal{J}(a) = a,
        \quad
        \mathcal{J}(x) = x,
        \quad
        \textrm{and}
        \quad
        \mathcal{J}(D) =  -\nabla_D
    \end{equation}
    for $a\in \A$, $x\in \Eps$ and $D\in \Der(\A)$, and extended to
    all degrees as homomorphism of $\wedge$.
\end{definition}
\begin{proposition}
   \begin{enumerate}
    \item
      The map $\mathcal{J}$ is a homomorphism of Poisson algebras.
    \item
      Let $\phi \in \Rothstein^r(\Eps)$ with $r\geq 2$, then
      \begin{equation}\label{eq:derived_bracket_formula}
         \mathcal{J}(\phi)(x_1,\ldots,x_{r-1}) =
         \RothBracket{\RothBracket{\ldots\RothBracket{\phi,x_1},\ldots},x_{r-1}}
      \end{equation}
      and
      \begin{equation}
         \sigma_{\mathcal{J}(\phi)}(x_1,\ldots,x_{r-2})a =
         \RothBracket{\RothBracket{\RothBracket{\ldots\RothBracket{\phi,x_1},
             \ldots},x_{r-2}},a} 
      \end{equation}
      for all $x_1,\ldots,x_{r-1} \in \Eps$ and $a\in \A$
   \end{enumerate}
\end{proposition}
\begin{proof}
   That $\mathcal{J}$ is an homomorphism is obviously true for
   generators and hence for all elements in
   $\Rothstein^\bullet(\Eps)$. To show
   Equation~(\ref{eq:derived_bracket_formula}), note first 
   that $[\mathcal{J}(\phi),x] = [\mathcal{J}(\phi),\mathcal{J}(x)] =
   \mathcal{J}(\RothBracket{\phi,x})$ for all $x\in \Eps$, and finish
   the proof then by induction over $r$. The rest is now clear.
\end{proof}
Recall that for a projective and finitely generated module $\Eps$ we have an
isomorphism $\lambda: \SDer^p(\A) \otimes_\A \Anti_\A^\bullet \Eps
\longrightarrow \SDer^p(\A,\Anti_\A^\bullet \Eps)$ given by
\[
\lambda(P \otimes \xi)(a_1,\ldots,a_p) = P(a_1\ldots,a_p)\xi
\]
for all $a_1,\ldots,a_p \in \A$. If now $\phi \in
\Sym_{\mathcal{A}}^p\Der(\mathcal{A}) \otimes_\A \Anti_\A^k\Eps$ for
$p \geq 1$ and $k \geq 0$ then
\[
\RothBracket{\RothBracket{\cdots\RothBracket{\phi,a_1},\cdots},a_p} =
(-1)^p \lambda(\phi)(a_1,\ldots,a_p)
\]
and injectivity of $\lambda$ implies that $\RothBracket{\phi,a} = 0$
for all $a\in \A$ if and only $\phi = 0$. 
\begin{lemma}\label{lemma:injectivity-Lambda}
    Let $\phi\in \Rothstein^r(\Eps)$ with $r\geq 1$. 
    Then
    \begin{equation}\label{eq:injectivity}
      \RothBracket{\RothBracket{\cdots\RothBracket{\phi,x_1},\cdots},x_r} = 0
    \end{equation}
    for all $x_1,\ldots,x_r \in \Eps$ if and only if $\phi = 0$.
\end{lemma}
\begin{proof}
    Since the bilinear form $\SP{\cdot,\cdot} =
    \RothBracket{\cdot,\cdot}|_{\Eps\times \Eps}$ is non-degenerate,
    the lemma is true for $r = 1$. For $\phi \in \Rothstein^r(\Eps)$
    with $r\geq 2$, we have $\RothBracket{\phi,x_1} \in
    \Rothstein^{r-1}(\Eps)$, and we get by induction that
    (\ref{eq:injectivity}) is true for all $x_1,\ldots,x_r \in \Eps$
    if and only if $\RothBracket{\phi,x} = 0$ for all $x\in \Eps$.
    Hence
    \[
    \RothBracket{\phi,\SP{x,y}} = \RothBracket{\phi  ,\RothBracket{ x,y}} =
    \RothBracket{\RothBracket{\phi,x},y} +(-1)^r 
    \RothBracket{x,\RothBracket{\phi,y}}  = 0
    \]
    for all $x,y \in \Eps$, and due to fullness we conclude that also
    $\RothBracket{\phi,a} = 0$ for all $a\in \A$.  Now write $\phi$ as
    a sum $\phi = \sum_{2p +k = r} \phi_p$ with $ \phi_p \in
    \Sym_{\mathcal{A}}^p\Der(\mathcal{A})\otimes_\A\Anti_\A^{r-2p}\Eps$.
    Then $\RothBracket{\phi_p,a}=0$ for all $p\geq 1$, and thus
    $\phi_p = 0$ for all $p \geq 1$. Hence we have
    $\RothBracket{\phi_0,x} = 0$ for all $x\in \Eps$, and with the
    non-degeneracy of the inner product follows then that also $\phi_0
    = 0$.
\end{proof}
The last lemma implies now immediately the injectivity of
$\mathcal{J}$. In general $\mathcal{J}$ is not
surjective, however. The reason is that there might appear
non-factorizing symmetric multiderivations of $\A$ when calculating
the higher symbols $\pi^{(p)}_\C$ of an element $\C \in
\MC^\bullet(\Eps)$ while the Rothstein-Poisson algebra was
constructed using only symmetric products of ordinary derivations.

\begin{corollary}\label{corollary:Isomorphism-Rothstein-Courant}
    Let $\hat{\MC}^\bullet(\Eps)$ be the $\wedge$-subalgebra of
    $\MC^\bullet(\Eps)$ generated by $\A$, $\Eps$ and $\MC^2(\Eps)$.
    Then $\hat{\MC}^\bullet(\Eps)$ is closed under the bracket
    $[\cdot,\cdot]$ and $\mathcal{J}$ is an isomorphism of Poisson
    algebras
    \begin{equation}
        \mathcal{J}: \Rothstein^\bullet(\Eps) \longrightarrow
        \hat{\MC}^\bullet(\Eps).
    \end{equation}
\end{corollary}
\begin{proof}
    From Lemma~\ref{lemma:injectivity-Lambda} follows that
    $\mathcal{J}$ is injective.  Moreover it is clear from the Leibniz
    rule (\ref{eq:courant-leibniz}) that $\hat{\MC}^\bullet(\Eps)$ is
    a Poisson subalgebra.  If $\mathsf{D} \in \MC^2(\Eps)$ we can
    define an element $\xi \in \Anti_\mathcal{A}^2 \Eps$ by
    $\SP{\xi,x\wedge y} = \SP{\mathsf{D}(x) -
      \nabla_{\sigma_{\mathsf{D}}} x,y}$.  It follows that
    $\RothBracket{- \xi +\sigma_{\mathsf{D}},x} = \mathsf{D}(x)$ for
    all $x\in \Eps$, hence $\mathsf{D} \in
    \mathcal{J}(\Rothstein^2(\Eps))$ and therefore $\MC^2(\Eps) \cong
    \Rothstein^2(\Eps)$. Since $\mathcal{J}$ is a homomorphism with
    respect to the $\wedge$-products, the rest follows now
    immediately.
\end{proof}
\begin{remark}\label{remark:Iso-Rothstein1-Rothstein2}
    It follows again that Rothstein algebras to different connections
    are isomorphic since they are all isomorphic to
    $\hat{\MC}^\bullet(\Eps)$.
\end{remark}
The following easy observation is crucial to simplify the deformation
theory of Courant algebroids drastically:
\begin{lemma}
    We have $\hat{\MC}^3(\Eps) = \MC^3(\Eps)$.
\end{lemma}
\begin{proof}
    Let $\C \in \MC^3(\Eps)$ and let $\dif_{\C} \in \Der(\A,\Eps)$ be
    given by $\SP{\dif_{\C} a,x} = \sigma_{\C}(x)a$. Since $\Eps$ is
    projective and finitely generated, we can find $D^1,\ldots,D^n \in
    \Der(\A)$ and $e_1,\ldots,e_n \in \Eps$ such that $\dif_{\C}(a) =
    D^i(a) e_i$. It follows that $ \sigma_{\C}(x)a = \SP{\dif_{\C}
      a,x} = D^i(a)\SP{e_i,x}$, i.e.  $\sigma_{\C}(x) =
    \SP{e_i,x}D^i$.  Let $\nabla$ be a metric connection for $\Eps$
    and define $T \in \MC^3(\Eps)$ by $T = \C - \nabla_{D^i} \wedge
    e_i$.  Then
    \begin{align*}
        \SP{ T(x,y),z} &= \SP{\C(x,y),z} - \SP{\nabla_{D^i} x,y}
        \SP{e_i,z} + \SP{\nabla_{D^i}x,z}\SP{ e_i,y} -
        \SP{\nabla_{D^i}y,z} \SP{ e_i,x} \\
        &= \SP{\C(x,y),z} - \SP{\nabla_{\sigma_{\C}(z)}x,y} +
        \SP{\nabla_{\sigma_{\C}(y)}x,z} -
        \SP{\nabla_{\sigma_{\C}(x)}y,z},
    \end{align*}
    and one easily shows that $\eta = \SP{T(\cdot,\cdot),\cdot}$
    is skew-symmetric and $\A$-linear.  Hence $\C \in
    \Anti_\mathcal{A}^3 \MC^1(\Eps) \oplus \big( \MC^1(\Eps)
    \wedge_\mathcal{A} \MC^2(\Eps)\big)$.
\end{proof}

Since $\mathcal{J}$ respects the $\wedge$-product we have also found
the inverse image of $\C \in \mathcal{C}^3(\mathcal{E})$ in the
Rothstein algebra: Use the strongly non-degenerate inner product to
define the element $\xi \in \Anti_\mathcal{A}^3 \Eps$
such that $\mathcal{J}(\xi) = T$. It follows then that
\begin{equation}
    \mathcal{J}( \xi - D^i \wedge e_i) =
    \mathcal{J}(\xi) 
    - \mathcal{J}(D^i)\wedge 
    \mathcal{J}(e_i) = T + \nabla_{D^i} \wedge e_i = \C.
\end{equation}

Let $m\in \MC^3(\Eps) = \hat{\MC}^3(\Eps)$ with $[m,m] = 0$. Then
$\delta_m = [m,\cdot]$ squares to zero and since
$\hat{\MC}^\bullet(\Eps)$ is closed under $[\cdot,\cdot]$ we get a
subcomplex
\begin{equation}
    \label{eq:RestrictedComplex}
    \xymatrix{ \A \ar[r]^-{\delta_m} & \Eps
      \ar[r]^-{\delta_m} & \hat{\MC}^2(\Eps) \ar[r]^-{\delta_m} &  \hat{\MC}^3(\Eps)
      \ar[r]^-{\delta_m} &  \hat{\MC}^4(\Eps)
      \ar[r]^-{\delta_m} &  }
    \cdots
\end{equation}
For $r\leq 3$ we have $\hat{\MC}^r(\Eps) = \MC^r(\Eps)$ and hence also
$H^r(\hat{\MC}(\Eps),\delta_m) = H^r(\MC(\Eps),\delta_m)$. Let $m_t =
m + m_1 t + m_2 t^2 +\cdots + m_k t^k$ be a deformation of $m$ of
order $k$. One can show that $ \sum_{i=1}^{k}[m_i,m_{k+1-i}]$ is a
cocycle and that moreover $m'_t = m_t + m_{k+1} t^{k+1}$ is a
deformation of order $k+1$ if and only if
\begin{equation}\label{eq:MaurerCartan}
    2 \delta_m m_{k+1} = -\sum_{i=1}^{k}[m_i,m_{k+1-i}].
\end{equation}
Since $\hat{\MC}^\bullet(\Eps)$ is closed under the bracket, the right
hand side of (\ref{eq:MaurerCartan}) is an element in
$\ker(\delta_m:\hat{\MC}^4(\Eps) \longrightarrow \hat{\MC}^5(\Eps))$,
whence the obstructions for finding $m_{k+1} \in \MC^3(\Eps) =
\hat{\MC}^3(\Eps)$ are in $H^4(\hat{\MC}(\Eps),\delta_m)$. To study
the deformation theory of a given Courant algebroid structure $m$, we
therefore can restrict ourselves to the smaller algebra
$\hat{\MC}^\bullet(\Eps)$ .
\begin{theorem}[Deformation theory, II]
    \label{theorem:DeformationProblem}
    The formal deformation theory of Courant algebroid structures is
    controlled by $\hat{\MC}^\bullet(\mathcal{E})$ and the relevant
    cohomologies for a given Courant algebroid structure $m$ are
    $H^\bullet(\hat{\MC}(\mathcal{E}), \delta_m)$. Equivalently, one
    can use the Rothstein algebra $\Rothstein^\bullet(\mathcal{E})$
    with the differential $\delta_\Theta = \RothBracket{\Theta,
      \cdot}$ instead, where $\Theta = \mathcal{J}^{-1}(m)$.
\end{theorem}
\begin{remark}
    \label{remark:WhyBetter}
    In view of the examples in \eqref{eq:BadC} at the beginning of
    Section~\ref{ThePoissonAlgebraC} we see that the complex
    $\hat{\MC}^\bullet(\mathcal{E})$ typically is strictly smaller
    than $\MC^\bullet(\mathcal{E})$. Thus the map $\mathcal{J}$ is
    only an injection but not surjective in general. Hence the deformation
    problem of Courant algebroids is simplified significantly by
    replacing $\MC^\bullet(\mathcal{E})$ with
    $\hat{\MC}^\bullet(\mathcal{E})$. This shows the advantage of the
    formulation of the deformation problem using the Rothstein algebra
    in Theorem~\ref{theorem:DeformationProblem} compared to the more
    naive version in Remark~\ref{remark:DeformationTheoryI}. Moreover,
    even in the case of smooth manifolds, where the map $\mathcal{J}$
    is an isomorphism, the approach with the Rothstein algebra seems
    to be simpler thanks to the easier characterization of the
    underlying $\mathcal{A}$-modules $\Rothstein^\bullet(\mathcal{E})$
    compared to $\mathcal{C}^\bullet(\mathcal{E})$.
\end{remark}

Suppose now in the following that $\SDer^p(\A) = \Sym_\A^p \Der(A)$
for all $p \geq 0$.  This condition is for example satisfied if
$\Der(\A)$ is a finitely generated and projective module over $\A$, an
assumption which we will have to make for the Fedosov construction in
the next section. Note that this implies that also $\SDer^p(\A,
\mathcal{F}) = \Sym_\A^p \Der(\A) \otimes_\A \mathcal{F}$ for any finitely
generated, projective module $\mathcal{F}$.

\begin{lemma}\label{lemma:tau->psi}
   Let $\tau: \A \longrightarrow \Rothstein^r(\Eps)$ be a derivation
   such that $\RothBracket{\tau(a),b} = \RothBracket{\tau(b),a}$ for
   all $a,b \in \A$. Then there exists an element $\psi \in
   \Rothstein^{r+2}(\Eps)$ with $\tau(a) = \RothBracket{\psi,a}$ for
   all $a\in \A$.
\end{lemma}
\begin{proof}
    It is sufficient to prove the lemma for the case that $\tau$ is a
    map $\tau: \A \longrightarrow \Sym^p_\A\Der(\A) \otimes_\A
    \Anti_\A^k \Eps$. In this case we may view $\tau$ also as
    symmetric multi-derivation $\A^{p+1} \longrightarrow \Anti_\A^k
    \Eps$, i.e.~$\tau \in \SDer^{p+1}(\A, \Anti_\A^k \Eps)$. Since
    $\Eps$ is finitely generated and projective, this defines now an
    element $\phi \in \Sym^{p+1}_\A\Der(\A) \otimes_\A \Anti_\A^k
    \Eps$, and one easily shows that $\RothBracket{\psi,a} = \tau(a)$
    for all $a\in \A$.
\end{proof}
\begin{proposition}
    Suppose that $\Sym_\A^\bullet\Der(\A)= \SDer^\bullet(\A)$, then
    the map $\mathcal{J}:\Rothstein^\bullet(\Eps) \longrightarrow
    \MC^\bullet(\Eps)$ is an isomorphism.
\end{proposition}
\begin{proof}
    We proof the proposition by induction over $r$. The cases $r =
    0,1$ are clear. Suppose that $r \geq 1$, and let $\C\in
    \MC^{r+1}(\Eps)$. By induction we get well-defined
    $\ring{R}$-linear maps
   \[
   \mu: \Eps   \longrightarrow \Rothstein^r(\Eps)\quad\text{such that }
   \mathcal{J}(\mu(x)) = i_x \C \quad\text{for all } x \in \Eps,
   \text{ and}
   \]
   \[
   \tau:  \A \longrightarrow \Rothstein^{r-1}(\Eps)\quad\text{such that }
   \mathcal{J}(\tau(a)) = \dif_\C a \quad\text{for all } a \in \A.
   \]
   %We have to show that there exist $\phi \in \Rothstein^{r+1}(\Eps)$
   %such that $\mu(x) = \RothBracket{\phi,x}$ for all $x\in \Eps$.
   Now
   $\tau$ is a derivation and satisfies
   \[
   \mathcal{J}(\RothBracket{\tau(a),b}) =
   [\mathcal{J}(\tau(a)),b] = [\dif_\C a,b] = [[\C,a],b] =
   [[\C,b],a] = [\dif_\C b,a] = \mathcal{J}(\RothBracket{\tau(b),a}).
   \]
   Since $\mathcal{J}$ is injective, Lemma~\ref{lemma:tau->psi}
   implies that $\tau(a) = \RothBracket{\psi,a}$ for some $\psi \in
   \Rothstein^{r+1}(\Eps)$. Consider now the map $\mathsf{H} = \C -
   \mathcal{J}(\psi)$. Then
   \begin{align*}
      \sigma_\mathsf{H}(x_1,\ldots,x_{r-1})a &=
      \RothBracket{\dif_\C(x_1,\ldots,x_{r-2})a,x_{r-1}} -
      \sigma_{\mathcal{J}(\psi)}(x_1,\ldots,x_{r-1})a\\
      &=
      \RothBracket{\ldots\RothBracket{\RothBracket{\psi,a},x_1}\ldots,x_{r-1}}
      -
      \RothBracket{\RothBracket{\ldots\RothBracket{\psi,x_1},\ldots,x_{r-1}},a}
      \\
      &=0,
   \end{align*}
   hence $\mathsf{H} \in \Anti_\A^r \Eps' \otimes_\A \Eps$ and there is a
   uniquely defined element $\xi \in \Anti_\A^{r+1}\Eps$ such that
   $\mathcal{J}(\xi) = \mathsf{H}$. It follows now that $\mathcal{J}(\xi+\psi)
   = \C$.
\end{proof}
In the case that $\A = C^\infty(M)$ for a smooth manifold $M$ the two
approaches to a deformation complex for Courant algebroids given by
$\MC^\bullet(\Eps)$ and $\Rothstein^\bullet(\Eps)$, respectively, are
hence completely equivalent.

%
% Fedosov construction for $\Rothstein^\bullet(\Eps)$
%
\section{Fedosov construction for $\Rothstein^\bullet(\Eps)$}
\label{section:Fedosov}

We will now briefly present a Fedosov construction for the
Rothstein-Poisson algebra. We mainly follow here Bordemann
\cite{bordemann:1996a,bordemann:2000a} who gave a Fedosov construction for the
Grassman algebra of sections of a smooth bundle over a symplectic
manifold, and adapt it into our algebraic setting.

Throughout this section we will always assume that the module $\dermod
= \Der(\A)$ of derivations of the algebra $\A$ is finitely generated
and projective. In many situations $\dermod$ behaves then like the
$C^\infty(M)$-module $\Gamma^\infty(M)$ for a smooth manifold $M$. For
example, identities like $\Hom_\A(\dermod,\dermod) = \dermod'\otimes
\dermod$ or $(\dermod\otimes\dermod)' = \dermod'\otimes \dermod'$
still remain valid in this setting. Here and in the following, unless
otherwise stated, all tensor products will be over the algebra $\A$.

Consider now the $\A$-module
\begin{equation}\label{eq:def_of_W_0}
   \W_0 = \prod_{p=0}^\infty \Sym^p\dermod' \otimes
   \Sym^\bullet \dermod \otimes \Anti^\bullet \Eps
\end{equation}
and the $\A[[\lambda]]$-modules 
\begin{align*}
   \W &= \W_0[[\lambda]] &&\text{and} & \W \otimes \Anti &= (\W_0
   \otimes \Anti^\bullet \dermod')[[\lambda]]
\end{align*}
for a formal parameter $\lambda$. On these modules we have obvious
$\A$- or $\A[[\lambda]]$-bilinear multiplications which we denote by
$\bullet$. We define further the $\A[[\lambda]]$-linear endomorphisms
$\degs'$, $\degs$, $\degE$ and $\dega$ of $\W\otimes \Anti$, which on
homogeneous elements
\[
v \in \Sym^p\dermod' \otimes \Sym^q \dermod
\otimes \Anti^k \Eps \otimes \Anti^\ell \dermod'
\]
are given by $\degs' v = p v$, $\degs v = q v$, $\degE v = k v$ and
$\dega v= \ell v$.
Moreover, we will use the total degree
\begin{equation*}
   \Deg = 2 \lambda \frac{\partial}{\partial\lambda} + \degs' + \degs + \degE,
\end{equation*}
as well as the  degree
\begin{equation*}
   \DegR =  2 \lambda \frac{\partial}{\partial\lambda} + 2 \degs + \degE.
\end{equation*}
Note that all these maps are derivations of $\bullet$.  We will also
need the parity operator $\parity\in \End_{\A[[\lambda]]}(\W\otimes \Anti) $
with respect to the degree $\degE$, i.e.~we set $\parity(v) = (-1)^k
v$ for homogeneous elements $v\in \W\otimes \Anti$ with $\degE v = k
v$.

Consider now the operators $\calQ,\calQ^\ast,\calS \in
\End_{\A[[\lambda]]}(\W\otimes\Anti \otimes_{\A[[\lambda]]}
\W\otimes\Anti)$ given on homogeneous elements by
\begin{equation*}\begin{split}
      \calQ\big((\alpha_1\vee&\ldots\vee \alpha_k \otimes X_1
      \vee\ldots \vee X_l \otimes x_1\wedge\ldots \wedge x_m \otimes \omega
      )\\
      &\qquad\quad \otimes ( \beta_1\vee\ldots\vee \beta_r
      \otimes Y_1 \vee\ldots\vee Y_s \otimes y_1\wedge \ldots \wedge
      y_t \otimes \eta)\big) \\ 
      & = \sum_{ i=1}^k \sum_{j=1}^s
      \alpha_i(Y_j) \alpha_1\vee\elide{i}\vee \alpha_k \otimes X_1\vee
      \ldots\vee X_l \otimes x_1\wedge \ldots\wedge x_m \otimes \omega \\[-3mm]
      &\qquad\qquad\qquad\qquad \otimes \beta_1\vee\ldots\vee \beta_r \otimes Y_1
      \vee \elide{j}\vee Y_s
      \otimes y_1\wedge \ldots\wedge y_t \otimes \eta,\\[3mm]
   \end{split}\end{equation*}
\begin{equation*}\begin{split}
      \calQ^*\big((\alpha_1&\vee\ldots\vee \alpha_k \otimes X_1
      \vee\ldots \vee X_l \otimes x_1\wedge\ldots \wedge x_m\otimes\omega
      )\\
      &\qquad \otimes ( \beta_1\vee\ldots\vee \beta_r
      \otimes Y_1 \vee\ldots\vee Y_s \otimes y_1\wedge \ldots \wedge
      y_t\otimes\eta)\big) \\ 
      &=\sum_{ i=1}^r \sum_{j=1}^l
      \beta_i(X_j) \alpha_1\vee\ldots\vee \alpha_k \otimes X_1\vee
      \elide{j}
      \vee X_l \otimes x_1\wedge \ldots\wedge x_m \otimes\omega\\[-3mm]
      &\quad\qquad\qquad\qquad\qquad \otimes \beta_1\vee\elide{i} \vee
      \beta_r \otimes 
      Y_1 \vee \ldots \vee Y_s
      \otimes y_1\wedge\ldots\wedge  y_t\otimes\eta\\[3mm]
   \end{split}
\end{equation*}
and
\begin{equation*}\begin{split}
      \calS\big((\alpha_1&\vee\ldots\vee \alpha_k \otimes X_1
      \vee\ldots \vee X_l \otimes x_1\wedge\ldots \wedge
      x_m\otimes\omega)\\
      &\qquad \otimes ( \beta_1\vee\ldots\vee \beta_r \otimes Y_1
      \vee\ldots\vee Y_s \otimes y_1\wedge \ldots \wedge
      y_t\otimes \eta)\big) \\
      & = \frac{(-1)^{m-1}}{2} \sum_{ i=1}^m \sum_{j=1}^t (-1)^{i+j}
      \SP{x_i,y_j} \alpha_1\vee\ldots\vee \alpha_k \otimes X_1\vee
      \otimes  \vee X_l \otimes x_1\wedge \elide{i} \wedge x_m \otimes
      \omega \\[-3mm]
      &\quad \qquad\qquad\qquad\qquad \otimes \beta_1\vee\ldots \vee
      \beta_r \otimes Y_1 \vee \ldots \vee Y_s \otimes
      y_1\wedge\elide{j} \wedge y_t \otimes \eta,
   \end{split}\end{equation*}
as well as the $\A[[\lambda]]$-linear operator $\Delta: \W\otimes\Anti
\longrightarrow \W\otimes\Anti$, given on homogeneous elements by
\begin{equation*}\begin{split}
      \Delta \big(\alpha_1\vee&\ldots\vee \alpha_k \otimes X_1
      \vee\ldots \vee X_l \otimes x_1\wedge\ldots \wedge x_m\otimes\omega
      \big)\\ 
      &= \sum_{ i=1}^k \sum_{j=1}^l
      \alpha_i(X_j) \alpha_1\vee\elide{i}\vee \alpha_k \otimes X_1\vee
      \elide{j}\vee X_l \otimes x_1\wedge \ldots\wedge x_m\otimes\omega.
   \end{split}\end{equation*}
Let further $\calQ_{12}$, $\calQ_{23}$, $\calQ_{13}$, $\calS_{12}$,
$\calS_{23}$, $\calS_{13}, \parity_2: \bigotimes_{\A[[\lambda]]}^3
(\W\otimes\Anti) \longrightarrow \bigotimes_{\A[[\lambda]]}^3
(\W\otimes \Anti)$ be given by
\begin{align*}
   \calQ_{12} &= \calS\otimes \id & \calQ_{23} &= \id \otimes \calS
   &\calQ_{13} &= (\id \otimes \switch{}) \circ (\calQ\otimes \id
   )\circ
   (\id \otimes \switch{} )\\
   \calS_{12} &= \calS\otimes \id & \calS_{23} &= \id \otimes \calS
   & \calS_{13} &= (\id \otimes \switch{}) \circ (\calS\otimes \id
   )\circ (\id \otimes \switch{} ),
\end{align*}
where $\switch{}: (\W\otimes\Anti)\otimes(\W\otimes\Anti)
\longrightarrow (\W\otimes\Anti)\otimes(\W\otimes\Anti) $ denotes
the (non-graded) switch operator. Finally, set $\parity_2 = \id
\otimes \parity\otimes \id$.
\begin{lemma}
   The above defined operators satisfy the following properties:
   \begin{enumerate}
    \item     \begin{align*}
         \calQ \circ (\bullet\otimes\id) &= (\bullet\otimes\id
         )\circ (\calQ_{13} + \calQ_{23})\\
         \calQ \circ (\id \otimes \bullet) &= (\id \otimes\bullet)
         \circ (\calQ_{12} + \calQ_{13})\\
         \calS \circ (\bullet \otimes \id ) &= (\bullet\otimes\id
         )\circ (\calS_{23} + \calS_{13}\circ \parity_2)\\
         \calS \circ (\id \otimes \bullet) &= (\id
         \otimes\bullet) \circ (\calS_{12} + \calS_{13}\circ \parity_2).
      \end{align*}
    \item
      $\calQ_{12}$, $\calQ_{23}$, $\calQ_{13}$,
      $\calS_{12}$, $\calS_{23}$ and $\calS_{13}\circ \parity_2$ commute
      pairwise.
    \item
      The operator $\Delta$ satisfies
      \[
      \Delta\circ\bullet = \bullet \circ (\Delta\otimes \id + \calQ +
      \calQ^* + \id\otimes \Delta).
      \]
      Moreover, $\calQ$, $\calQ^*$, $\Delta\otimes \id$ and
      $\id\otimes \Delta$ commute pairwise.
   \end{enumerate}
\end{lemma}
\begin{proof}
   This follows easily from the definitions of the involved maps. %$\calQ$,
   % $\calQ^*$, $\calS$ and $\Delta$.
\end{proof}
\begin{proposition}
   Let $\kappa \in \ring{R}$, then
   \begin{equation}\label{eq:Definition_of_deformed_product}
      \bullet_\lambda^\kappa = \bullet \circ \E^{\lambda((1-\kappa)\calQ
        -\kappa \calQ^\ast+ \calS)} =
      \bullet \circ \sum_{n=0}^\infty\frac{\lambda^n}{n!}((1-\kappa)\calQ
      -\kappa \calQ^\ast+ \calS)^n
   \end{equation}
   defines an associative, $\A[[\lambda]]$-linear deformation of $\bullet$.
\end{proposition}
\begin{proof}
   For the proof consider first the case $\kappa = 0$ and set
   $\bullet_\lambda = \bullet \circ \E^{\lambda(\calQ +
     \calS)}$. Obviously, $\bullet_\lambda$ is $\A[[\lambda]]$-linear
   and $\bullet_0 = \bullet$. With the help of the previous lemma, a
   straight forward calculation shows that $\bullet_\lambda$ is
   associative. Let now $N_\kappa = \E^{\lambda \kappa\Delta}:
   \W\otimes\Anti \longrightarrow \W\otimes \Anti$ for $\kappa \in
   \ring{R}$ be the Neumaier operator \cite{neumaier:1998a}. Then
   $N_0 = \id$, and $N_\kappa$ is $\A[[\lambda]]$-linear and
   invertible with inverse $N_\kappa^{-1} = N_{-\kappa}$.
   Define $\bullet_\lambda^\kappa = N_{-\kappa}\circ
   \bullet_\lambda \circ ( N_\kappa \otimes N_\kappa)$. Clearly,
   $\bullet_\lambda^\kappa$ is also an associative deformation of
   $\bullet$, and another little calculation shows then that
   $\bullet_\lambda^\kappa$ is indeed given by Equation
   \eqref{eq:Definition_of_deformed_product}
\end{proof}
Note that the maps $\dega$, $\Deg$ and $\DegR$ are derivations of the
deformed product $\bullet^\kappa_\lambda$ as well.  

Define for homogeneous elements $\omega,\eta \in \W\otimes \Anti$ with
$\degE \omega = k \omega$, $\degE \eta = l \eta$, $\dega \omega = m
\omega$, $\dega \eta = n \eta$ their commutator by
\begin{equation}
   [\omega,\eta]_\kappa = \ad_\kappa(\omega)\eta =
   \omega\bullet_\lambda^\kappa  \eta -
   (-1)^{kl+mn} \eta \bullet_\lambda^\kappa \omega.
\end{equation}
Then the zeroth order of the commutator vanishes, thus for every
$\omega \in \W\otimes \Anti$ the map
$\frac{1}{\lambda}\ad_\kappa(\omega) = \frac{1}{\lambda}
[\omega,\cdot]_\kappa$ is a well-defined $\A[[\lambda]]$-linear
derivation of $\bullet_\lambda^\kappa$.
\begin{lemma}\label{lemma:weyl-type}
   Let $\omega,\eta \in \W_0\otimes \Anti$, then for $\kappa =
   \frac{1}{2}$ there appear only terms of odd $\lambda$-order in the
   commutator
   $[\omega,\eta]_\frac{1}{2}$.
\end{lemma}
\begin{proof}
   This follows easily by a short calculation.
\end{proof}
Define next the map $\delta: \W\otimes\Anti^k \longrightarrow
\W\otimes\Anti^{k+1}$ by the following conditions:
\begin{enumerate}
 \item $\delta (f \otimes 1\otimes 1\otimes 1 ) = 1\otimes 1\otimes
   1\otimes f$ for $f \in \dermod'$.
 \item $\delta v = 0$ for any homogeneous element $v\in
   \W\otimes\Anti$ with $\degs' v = 0$.
 \item $\delta$ is a $\dega$-graded derivation of $\bullet$.
\end{enumerate}
Similarly, define a $\dega$-graded derivation $\delta^*:
\W\otimes\Anti^k \longrightarrow \W\otimes\Anti^{k-1}$ of $\bullet$ by
\begin{equation}
   \delta^* (1 \otimes 1\otimes 1\otimes \alpha ) =  \alpha\otimes
   1\otimes  1\otimes 1
\end{equation}
for $\alpha \in \dermod'$, and as zero on elements $v\in \W$.
Moreover,  on homogeneous elements $\omega \in
\W\otimes\Anti$ with $\degs'\omega = k \omega$ and $\dega\omega = l
\omega$ we define the map  $\delta^{-1}$ by
\begin{equation}
   \delta^{-1}\omega = \begin{cases} \frac{1}{k+l} \delta^*\omega &
      \text{if } k+l \neq 0\\
      0 & \text{if } k+l = 0,
   \end{cases}
\end{equation}
and the map $\sigma$  by
\begin{equation}\label{eq.def_sigma}
   \sigma(\omega) = \begin{cases}  \omega &
      \text{if } k+l = 0\\
      0 & \text{if } k+l \neq 0.
   \end{cases}
\end{equation}

\begin{proposition}\label{prop:props_of_delta}
   We have $\delta^2 = (\delta^*)^2 = (\delta^{-1})^2 = 0$, and
   $\delta\delta^{-1} + \delta^{-1}\delta + \sigma = \id$. Moreover,
   $\delta$ is a $\dega$-graded derivation of $\bullet_\lambda^\kappa$.
\end{proposition}
\begin{proof}
   Follows by straight forward calculations.
\end{proof}

Consider now the $\A$-module $\Omega^k(\dermod,\W_0) = \Hom_\A(\Anti^k
\dermod, \W_0)$.  Since we assume that $\dermod$ is finitely generated
and projective, there is an canonical isomorphism
\begin{equation}\label{eq:Iso_WLambda0-Omega0}
   \W_0 \otimes \Anti^\bullet \dermod' \cong
   \Omega^\bullet(\dermod,\W_0)  
\end{equation}
of $\A$-modules. We further have canonical isomorphisms of
$\A[[\lambda]]$-modules
\begin{equation}\label{eq:Iso_WLambda-Omega}
   \W \otimes \Anti \cong \Omega^\bullet(\dermod,\W_0)[[\lambda]]  \cong
   \Omega^\bullet(\dermod,\W_0[[\lambda]]).
\end{equation}
In the following we may therefore identify $\W_0 \otimes \Anti$ with
$\Omega_0 = \Omega^\bullet(\dermod,\W_0)$, and $\W\otimes \Lambda$
with $\Omega = \Omega^\bullet(\dermod,\W)$.  The multiplication
$\bullet_\lambda^\kappa$ on $\W \otimes \Anti$ becomes under this
identification a multiplication $\wedge_\lambda^\kappa$ on $\Omega^\bullet$,
which is for $X_1,\ldots, X_{k+l} \in \dermod$ given by
\begin{equation}\begin{split}
      \omega \wedge_\lambda^\kappa \eta &(X_1,X_2,\ldots, X_{k+l}) \\
      &=   \sum_{\pi \in \Sym_{k,l}} \sign(\pi) 
      \omega(X_{\pi(1)},\ldots,X_{\pi(k)})\bullet_\lambda^\kappa \eta 
      (X_{\pi(k+1)},\ldots,X_{\pi(k+l)}) .
   \end{split}
\end{equation}
Choose now a torsion-free connection $\nabla^{\dermod}$ for $\dermod$,
and a metric connection $\nabla^\Eps $ for $\Eps$. This defines us
also a connection $\nabla$ on $\W_0\otimes \Anti\cong \Omega_0$ by
extension using the (non-graded) Leibniz rule.  In the following we
consider the curvature $R^{\dermod} $ of $\nabla^\dermod$ as an
element in $\W_0\otimes\Anti^2 \cong \Omega^2_0$ with $\degs'
R^{\dermod} = R^{\dermod}$, $\degs R^{\dermod} = R^{\dermod}$ and
$\degE R^{\dermod} = 0$.  Moreover, let $r^\Eps \in \Omega^2_0$ be
defined by
\begin{equation}
   \SP{r^\Eps(X,Y),x\wedge y} = \SP{R^\Eps(X,Y)x,y},
\end{equation}
where $R^\Eps$ denotes the curvature of $\nabla^\Eps$. Then $\degs'
r^\Eps = \degs r^\Eps = 0$ and $\degE r^\Eps = 2 r^\Eps$. We hence
have $\Deg R^\dermod = \DegR R^\dermod = 2 R^\dermod$ and $\Deg r^\Eps
= \DegR r^\Eps = 2 r^\Eps$, and thus
\begin{equation*}
   R = R^{\dermod} - r^\Eps
\end{equation*}
is a homogeneous element of degree 2 with respect
to the both degrees $\Deg$ and $\DegR$.

Let the map $D: \Omega^k
\longrightarrow \Omega^{k+1}$ be defined by
\begin{equation}
   D\omega(X_0,\ldots,X_k) = \sum_{i=0}^k (-1)^i (\nabla_{X_i}
   \omega)(X_0,\elide{i}, X_k).
\end{equation}
for $X_0,\ldots, X_k \in \dermod$.  Using the isomorphism
(\ref{eq:Iso_WLambda-Omega}) we may also consider $D$ as a map
$\W\otimes \Lambda^k \longrightarrow \W \otimes \Lambda^{k+1}$.
\begin{lemma}\label{lemma:rules_for_D}
   \begin{enumerate}
    \item
      $D$ is a $\dega$-graded derivation of $\bullet$ and
      $\bullet_\lambda^\kappa$. 
    \item
      We have $D \delta + \delta D = 0$, $\delta R = 0$ and $D R = 0$.
    \item
      We have $D^2 = \frac{1}{2}[D,D] = \frac{1}{\lambda}\ad_\kappa(R)$.
   \end{enumerate}
\end{lemma}  
\begin{proof}
    That $D$ is a derivation of $\bullet$ follows immediately from the
    derivation properties of $\nabla$. By a calculation on homogeneous
    elements one further shows that
   \[
   [D \otimes \id + (-1)^k \id\otimes
   D,(1-\kappa) \calQ - \kappa\calQ^\ast +\calS](\omega\otimes\eta) =
   0
   \]
   for all $\omega\in \W \otimes \Lambda$ and $\eta \in \W \otimes
   \Lambda^k$. Using this, the derivation property of $D$ with respect
   to $\bullet_\lambda^\kappa$ follows then immediately from the
   definition of $\bullet_\lambda^\kappa$ by
   Equation~(\ref{eq:Definition_of_deformed_product}).  The rest is
   now again a straight forward calculation, where one can restrict
   oneself to generators of the algebra $\W_0\otimes\Anti$ since all
   involved maps are derivations of $\bullet$.  Especially, $\delta R
   = 0$ follows from the first Bianchi identity for $R^\dermod$, while
   $DR = 0$ follows from the second Bianchi identity for both
   $R^\dermod$ and $r^\Eps$.
\end{proof}

Consider now for some $r_\kappa\in \W\otimes \Anti^1$ with $\parity r_\kappa =
r_\kappa$ the map
\begin{equation}
   \fedosovD_\kappa = -\delta + D +\frac{1}{\lambda}\ad_\kappa(r_\kappa).
\end{equation}
\begin{proposition}\label{prop:props_of_FedosovD}
   Let $\fedosovD_\kappa$ and $r_\kappa$ be defined as above, then it follows
   that
   \begin{equation}\label{eq:prop_FD1}
      \fedosovD_\kappa^2 = \frac{1}{\lambda}\ad_\kappa(-\delta r_\kappa +
      D r_\kappa + R +\frac{1}{\lambda} r_\kappa
      \bullet_\lambda^\kappa r_\kappa)
   \end{equation}
   and
   \begin{equation}\label{eq:prop_FD2}
      \fedosovD_\kappa(-\delta r_\kappa + D r_\kappa + R
      +\frac{1}{\lambda} r_\kappa 
      \bullet_\lambda^\kappa r_\kappa) = 0.
   \end{equation}
\end{proposition}
\begin{proof}
   These are immediate consequences of Lemma~\ref{lemma:rules_for_D} .
\end{proof}

\begin{theorem}\label{theorem:existence_of_r}
   There exists a unique element $r_\kappa \in \W\otimes \Anti^1$ with
   $\parity r_\kappa = r_\kappa$ such that
  \begin{align}\label{eq:prop_of_r}
     \delta^{-1} r_\kappa &= 0 & \text{and} &&\delta r_\kappa &= D
     r_\kappa + R +\frac{1}{\lambda} r_\kappa \bullet_\lambda^\kappa
     r_\kappa.     
  \end{align}
  For this $r_\kappa$, the Fedosov derivation $\fedosovD_\kappa =
  -\delta^{-1} + D + \frac{1}{\lambda} \ad_\kappa(r_\kappa)$ satisfies
  $\fedosovD_\kappa^2=0$. Explicitly, as a formal sum in the total degree
  $\Deg$, $r_\kappa$ is recursively defined by $r_\kappa^{(0)} =
  r_\kappa^{(1)} = r_\kappa^{(2)} = 0$, $r_\kappa^{(3)} =
  \delta^{-1}R$, and  for $ k\geq 1$ by
  \begin{equation}\label{eq:recursion_for_r}
    r_\kappa^{(k+3)} = \delta^{-1}\bigg(D r_\kappa^{(k+2)} +
    \frac{1}{\lambda}\sum_{l=1}^{k-1} r_\kappa^{(l+2)}\bullet_\lambda^\kappa
    r_\kappa^{(k+2-l)}\bigg).
 \end{equation}
 Moreover, $r_\kappa$ satisfies $\DegR r_\kappa^{(k)} = 2
 r_\kappa^{(k)}$ for all $k \geq 0$. 
\end{theorem}
\begin{proof}
   Since we suppose that $r_\kappa \in \W\otimes \Anti^1$ and
   $\delta^{-1} r_\kappa = 0$, it follows from
   Proposition~\ref{prop:props_of_delta} that $r_\kappa = \delta^{-1}
   \delta r_\kappa$, and hence $r_\kappa = \delta^{-1}(Dr_\kappa + R
   +\frac 1 \lambda r_\kappa \bullet^\kappa_\lambda r_\kappa)$. This
   proves uniqueness and gives us also the recursion formula
   \eqref{eq:recursion_for_r}. To show that the thereby defined
   element $r_\kappa$ indeed satisfies \eqref{eq:prop_of_r} we refer
   to \cite[Theorem 1.1]{bordemann:2000a}. Finally, that
   $r^{(k)}_\kappa$ is of degree 2 with respect to $\DegR$ follows now
   easily by induction over $k$.
\end{proof}

\begin{remark}\label{remark:r-weyl}
  Since all $r^{(k)}_\kappa$ have odd $\dega$-degree and even
  parity $\parity r_\kappa^{(k)} = r_\kappa^{(k)}$, we can write the
  recursion formula (\ref{eq:recursion_for_r}) also as
 \begin{equation*}
    r_\kappa^{(k+3)} = \delta^{-1}\bigg(D r_\kappa^{(k+2)} +
    \frac{1}{2\lambda}\sum_{l=1}^{k-1} [r_\kappa^{(l+2)},
    r_\kappa^{(k+2-l)}]_\kappa\bigg). 
  \end{equation*}
  With the help of Lemma~\ref{lemma:weyl-type} follows then that for
  $\kappa= \frac{1}{2}$ the element $r_{\scriptscriptstyle{\frac{1}{2}}}$ has
  only contributions of even $\lambda$-degree.
\end{remark}

Let for $s\in\mathbb{N}$ the subspaces $\W_{(s)} \otimes \Anti \subset
\W\otimes \Anti$ be defined as
\begin{equation}
   \W_{(s)} \otimes \Anti = \prod_{n\geq s\, }\bigoplus_{\, \ell+p = n}
   \lambda^\ell \Sym^p 
   \dermod' \otimes \Sym^\bullet \dermod \otimes \Anti^\bullet\Eps
   \otimes \Anti^\bullet \dermod',
\end{equation}
so $\W\otimes \Anti  = \W_{(0)}\otimes\Anti \supset \W_{(1)}\otimes \Anti
\supset \W_{(2)}\otimes \Anti \supset \ldots$. 
\begin{lemma}\label{lemma:Properties_of_W_(k)}
   \begin{enumerate}
    \item $ (\W_{(s)} \otimes \Anti) \bullet_\lambda^\kappa (\W_{(t)} \otimes
      \Anti) \subseteq \W_{(s+t)} \otimes \Anti\:$ for all $s,t \geq 0$.
    \item Let $r_\kappa$ be the element defined in
      Theorem~\ref{theorem:existence_of_r}, then $r_\kappa^{(3)} =
      \delta^{-1} R \in \W_{(1)} \otimes \Anti^1$, and
      $r_\kappa^{(3+k)} \in \W_{(2)} \otimes \Anti^1$ for all $k\geq
      1$.
    \item
      $\sigma(\W_{(s)} \otimes \Anti) \subseteq \lambda^s
      \Rothstein^\bullet(\Eps)[[\lambda]]$ for all $s \geq 0$.
   \end{enumerate}
\end{lemma}
\begin{proof}
   The first part is an immediate consequence of the definition of
   $\bullet_\lambda^\kappa$, while the second part follows easily by
   induction over $k$ using the recursion formula
   \eqref{eq:recursion_for_r}. The last part is clear.
\end{proof}
Consider now the operator $A_\kappa = [\delta^{-1},
D+\frac{1}{\lambda}\ad_\kappa(r_\kappa)]$, where the element $r_\kappa\in
\W_{(1)}\otimes \Anti^1 $ is defined by Theorem
\ref{theorem:existence_of_r}.
\begin{proposition}
   \begin{enumerate}
    \item $A_\kappa$ commutes with $\DegR$ and satisfies
      $A_\kappa(\W_{(s)} \otimes \Anti)\subseteq \W_{(s+1)} \otimes
      \Anti$ for all $s \geq 0$.
    \item
      We have a well-defined
      operator
      \[ \frac{1}{\id-A_\kappa} = \sum_{n=0}^\infty A_\kappa^n.\]
    \item Let the operator $\fedosovD_\kappa^{-1}$ be defined by
      \begin{equation}
         \fedosovD_\kappa^{-1} = -\delta^{-1} \frac{1}{\id -A_\kappa},
      \end{equation}
      then
      \begin{equation}\label{eq:D_homotopy}
         \fedosovD_\kappa\fedosovD_\kappa^{-1} +
         \fedosovD_\kappa^{-1}\fedosovD_\kappa + \frac{1}{\id- A_\kappa}
         \sigma = \id. 
      \end{equation}
    \item We have $\W \subseteq \ker \fedosovD_\kappa^{-1}$.
   \end{enumerate}
\end{proposition}
\begin{proof}
    The first part of the Proposition follows by a short
    calculation. For the second part, note that $A_\kappa$ raises the
    total degree at least by one, hence $\sum_{n=0}^\infty A_\kappa^n$
    is a well defined formal power series in the total degree. The
    third part as again an easy calculation, we refer here to
    \cite[Sect.~4]{waldmann:2007a} for the details. The last part
    finally follows from $[\delta^{-1},A_\kappa] = 0$.
\end{proof}
The next corollary is again standard for the Fedosov construction.
\begin{corollary}
   \begin{enumerate}
    \item Let $w \in \W$, then $\fedosovD_\kappa w = 0$ if and only if $w =
      \frac{1}{\id - A_\kappa} \sigma(w)$.
    \item The projection $\sigma$ defined by Equation
      \eqref{eq.def_sigma} restricts to an isomorphism
      \begin{equation}
         \sigma: \ker \fedosovD_\kappa \cap \W \longrightarrow
         \Rothstein^\bullet(\Eps)[[\lambda]] 
      \end{equation}
      with inverse
      \begin{equation}
         \tau_\kappa =  \frac{1}{\id -A_\kappa}.
      \end{equation}
   \end{enumerate}
\end{corollary}
Since $\fedosovD_\kappa$ is a derivation of $\bullet_\lambda^\kappa$, the
subspace $\ker \fedosovD_\kappa$ is closed under
$\bullet_\lambda^\kappa$. Further, $\W$ is also closed under
$\bullet_\lambda^\kappa$ and hence the subspace $\ker \fedosovD_\kappa \cap
\W$ as well. We can therefore make the following definition.
\begin{definition}\label{def:fedosov-star-product}
   The Fedosov star product on
   $\Rothstein^\bullet(\Eps)[[\lambda]]$ is 
   defined by
   \begin{equation}
      \phi\star_\kappa \psi = \sigma\big(
      \tau_\kappa(\phi)\bullet_\lambda^\kappa
      \tau_\kappa(\psi)\big) 
   \end{equation}
   for $\phi$, $\psi \in \Rothstein^\bullet(\Eps)[[\lambda]]$.
\end{definition}
\begin{remark}
   Since $\DegR = 2 \lambda \frac{\partial}{\partial\lambda} + 2 \degs +
   \degE$ commutes with $\sigma$ and $\tau_\kappa$, $\DegR$ is a
   derivation of $\star_\kappa$.  It follows that for $\phi,\psi \in
   \Rothstein^\bullet(\Eps)$ the star product $\phi\star \psi$ is
   given by a finite sum.
\end{remark}
One easily shows that $\star_\kappa$ is an associative deformation of
the graded commutative algebra $\Rothstein^\bullet(\Eps)$.
Hence the first order term of the $\star_\kappa$-commutator defines a
graded Poisson bracket on $\Rothstein^\bullet(\Eps)$.  In order to
calculate this bracket, which we expect to be the Rothstein-Poisson
bracket for the connection $\nabla^\Eps$, we will need the following
lemma.
\begin{lemma}\label{lemma:star_on_generators}
   For generators $a,b\in \A$, $x,y \in \Eps$ and $X,Y\in \dermod$ of
   $\Rothstein^\bullet(\Eps)$ we have the formulas
   \begin{align*}
      a \star_\kappa b &= ab,\qquad\qquad a \star_\kappa x = ax = x
      \star_\kappa a,&
      x \star_\kappa y &=  x\wedge y + \frac{\lambda}{2}\SP{x,y},\\[3mm]
      a \star_\kappa X &= a X + \lambda(1-\kappa) X(a), &
      X \star_\kappa a &= a X - \lambda\kappa X(a),\\[3mm]
      x \star_\kappa X &= X \otimes x + \lambda(1-\kappa)
      \nabla^\Eps_X x,& X \star_\kappa x &= X \otimes x -\lambda\kappa
      \nabla^\Eps_X x,
   \end{align*} 
   and
   \[
   X \star_\kappa Y = X \vee Y + \lambda(1-\kappa) \nabla^\dermod_Y
   X - \lambda\kappa \nabla^\dermod_X Y - \frac{\lambda}{2} r^\Eps(X,Y) -
   \lambda^2 (1-\kappa)\kappa\, s(X,Y),
   \]
   where $s(X,Y) = s(Y,X)$ is the trace of the endomorphism $Z \mapsto
   \nabla_{\nabla_Z X} Y$.
\end{lemma}
\begin{proof}
   First note that $\star_\kappa$ is $\DegR$-homogeneous, hence there
   can be no higher order terms as the ones given in the formulas
   above.  If one recalls that only that terms in
   $\tau(\phi)\bullet_\lambda^\kappa \tau(\psi)$ which belong to
   $(\W\setminus \W_{(2)}) \otimes \Anti$ can contribute to the first
   $\lambda$-order of $\phi\star_\kappa\psi =
   \sigma(\tau_\kappa(\phi)\bullet_\lambda^\kappa \tau_\kappa(\psi))$,
   the identification of the first order terms of the star product
   reduces to a few straight forward
   calculations. (Lemma~\ref{lemma:Properties_of_W_(k)} is very
   helpful here.) With a similar consideration we can calculate the
   second order term of $X \star_\kappa Y$.
\end{proof}
Using the last lemma it follows now immediately that the first order
term of the $\star_\kappa$-commutator coincides on generators and
hence on the whole algebra $\Rothstein^\bullet(\Eps)$ with the
Rothstein-Poisson bracket defined in
Theorem~\ref{theorem:RothsteinPoissonBracket}.  For $\phi\in
\Rothstein^r(\Eps)$ and $\psi\in \Rothstein^s(\Eps)$ we therefore have
\begin{equation}
   [\phi,\psi]_{\star_\kappa}= \phi \star_\kappa \psi - (-1)^{rs}
   \psi\star_\kappa \phi = \lambda \RothBracket{\phi,\psi} + \ldots
\end{equation}
for all parameter values $\kappa \in \ring{R}$.  We summarize our
results in the following main theorem of this section.
\begin{theorem}
  The Fedosov star product $\star_\kappa$ is a deformation
  quantization of the graded Poisson algebra
  $(\Rothstein^\bullet(\Eps),\RothBracket{\cdot,\cdot},\wedge)$.
\end{theorem}
  
We finish this section with the following
lemma, which says that for $\kappa = \frac{1}{2}$ the Fedosov star
product is of ``Weyl-type''.
\begin{lemma}\label{lemma:Fedosov-Weyl-Typ}
  Let $\phi \in \Rothstein^r(\Eps)$ and $\psi \in \Rothstein^s(\Eps)$,
  then for $\kappa = \frac{1}{2}$ there appear only terms of odd
  $\lambda$-order in the graded commutator
  $[\phi,\psi]_{\star_{\frac 1 2}}  =
  \phi\star_{\scriptscriptstyle{\frac{1}{2}}} \psi - (-1)^{rs} \psi
  \star_{\scriptscriptstyle{\frac{1}{2}}} \phi$.
\end{lemma}
\begin{proof}
   This follows from Lemma~\ref{lemma:weyl-type} and Remark~\ref{remark:r-weyl}.
\end{proof}

%
% Quantization of Courant-Algebroids
%
\section{Quantization of Courant-Algebroids}
\label{section:QuantizationOfCourantAlgebroids}

The star product on $\Rothstein^\bullet(\Eps)$ can now be used to
define the quantization of Courant algebroids. Let
$\Rothstein^\bullet_-(\Eps) = \bigoplus_{r=0}^\infty
\Rothstein^{2r+1}(\Eps)$ denote the subspace of odd elements. The
idea is to consider a formal sum $\Theta = \Theta_0 + \lambda
\Theta_1 + \cdots \in \Rothstein^\bullet_-(\Eps)[[\lambda]]$ such
that $\Theta_0$ defines a Courant structure, and then using the star
product to generate conditions on the higher order terms. The first
guess is hence the following:
\begin{definition}[Quantization of Courant algebroids, first
  version]\label{def:quantisierungI}
  Let $\Eps$ be a Courant algebroid over $\A$ with Courant structure
  $m \in \MC^3(\Eps)$, and let $\Theta_0\in \Rothstein^3(\Eps)$ be the
  corresponding homological element in the Rothstein-Poisson algebra,
  i.e.\ $\mathcal{J}(\Theta_0) = m$. A quantization of the Courant
  structure on $\Eps$ is a formal series $\Theta = \Theta_0 +\lambda
  \Theta_1 + \cdots $ with $\Theta_k \in \Rothstein_-^\bullet(\Eps)$
  for all $k \geq 0$, such that $\Theta \star_\kappa \Theta = 0$.
\end{definition}
In Section~\ref{section:Lie-Rinehart-Pairs} we will see that the
double of a Lie-Rinehart pair provides a non-trivial example where this
quantization condition is already fulfilled for $\Theta_0$, i.e.\
where $\Theta_0 \star_\kappa \Theta_0 = 0$ without higher order
corrections.

In general there will be obstructions for the existence of a
quantization. Write the star product as a formal sum
\[
\phi \star_\kappa \psi = \phi\wedge \psi + \sum_{k=1}^\infty
\lambda^k C^\kappa_k(\phi,\psi)
\]
with $\phi,\psi \in \Rothstein^\bullet(\Eps)$. This defines us for $k \geq 1$ the
graded and 
$\ring{R}[[\lambda]]$-bilinear maps
\[
C^\kappa_k:
\Rothstein^r(\Eps)[[\lambda]] \times \Rothstein^s(\Eps)[[\lambda]]
\longrightarrow \Rothstein^{r+s-2k}(\Eps)[[\lambda]],
\]
where the specified grading is implied by the fact that $\DegR$ is a
derivation of $\star_\kappa$.  Let now $\Theta = \Theta_0 +\lambda
\Theta_1 + \cdots \in \Rothstein_-^\bullet(\Eps)[[\lambda]]$ with
$\Theta_0 \in \Rothstein^3(\Eps)$ satisfying
$\RothBracket{\Theta_0,\Theta_0}=0$. We have up to order three in
$\lambda$ that
\begin{equation}
   \begin{split}\label{eq:Theta-star-Theta}
      \Theta \star_\kappa \Theta &= \lambda^2
      \big(\RothBracket{\Theta_0,\Theta_1} + C^\kappa_2(\Theta_0,\Theta_0)\big) \\
      &\qquad + \lambda^3 \big(\RothBracket{\Theta_0,\Theta_2} +
      \frac{1}{2}\RothBracket{\Theta_1,\Theta_1} +
      C_2^\kappa(\Theta_0,\Theta_1) + C_2^\kappa(\Theta_1,\Theta_0) +
      \C_3^\kappa(\Theta_0,\Theta_0) \big) + \cdots\\
      &= \lambda^2 (\delta_{\Theta_0} \Theta_1 +
      C^\kappa_2(\Theta_0,\Theta_0)) + \lambda^3 (\delta_{\Theta_0}
      \Theta_2 + T_3)+ \cdots
   \end{split}
\end{equation}
with $T_3 = \frac{1}{2}\RothBracket{\Theta_1,\Theta_1} +
C_2^\kappa(\Theta_0,\Theta_1) + C_2^\kappa(\Theta_1,\Theta_0) +
\C_3^\kappa(\Theta_0,\Theta_0)$. Recall now that
$[\Theta,[\Theta,\Theta]_{\star_\kappa}]_{\star_\kappa} = 0$ for an
arbitrary odd element $\Theta \in
\Rothstein_-^\bullet(\Eps)[[\lambda]]$ by associativity of
$\star_\kappa$. In third $\lambda$-order this yields
$\delta_{\Theta_0} C_2^\kappa(\Theta_0,\Theta_0) = 0$, hence
$\Theta\star_\kappa \Theta=0$ can be satisfied up to order $2$, if and
only if the closed element $C_2^\kappa(\Theta_0,\Theta_0) $ is
exact. As a consequence of the grading, solutions of the equation
$\delta_{\Theta_0} \Theta_1 = -C_2^\kappa(\Theta_0,\Theta_0)$ are, up
to cocycles, elements $\Theta_1 \in \Rothstein^1(\Eps) = \Eps$.
The obstructions for solving this equation are therefore in
$H^2(\Rothstein(\Eps),\delta_{\Theta_0}) \cong H^2_m(\MC(\Eps))$.

Suppose now that we have already found $\Theta_1 \in \Eps$ with
$\delta_{\Theta_0} \Theta_1 = -C_2^\kappa(\Theta_0,\Theta_0)$. In this
case the fourth $\lambda$-order in
$[\Theta,[\Theta,\Theta]_{\star_\kappa}]_{\star_\kappa} = 0$ shows
then that also $\delta_{\Theta_0} T_3 = 0$. However, since $\Theta_1
\in \Eps$ we have $T_3 \in \Rothstein^0(\Eps) = \A$, so the equation
$\delta_{\Theta_0} \Theta_2 = -T_3$ has a solution if and only if $T_3
= 0$.  In general we therefore can not expect that a quantization as
proposed in Definition~\ref{def:quantisierungI} exists.

In fact, Example~\ref{example:semisimpleLie} provides a whole class of
Courant algebroids for which such a quantization does not exist:
Consider a finite-dimensional semi-simple Lie algebra $\mathfrak{g}$
over a field $\kk$ with characteristic zero. In this case $\Der(\A) =
\{0\}$, so the Rothstein-Poisson algebra is just the Grassmann algebra
$\Anti^\bullet\mathfrak{g}$ of the Lie algebra $\mathfrak{g}$, and the
Fedosov star product is given by
\[
\xi \star \zeta = \wedge\circ \E^{\lambda \calS}(\xi \otimes
\zeta)\quad \text{for }\xi,\zeta \in \Anti^\bullet \mathfrak{g}.
\]
One easily shows that the element $\Theta_0\in \Anti^3 \mathfrak{g}$
which corresponds to the Lie algebra structure on $\mathfrak{g}$ is
given by
\begin{equation}\label{eq:Theta_for_LieAlg}
\Theta_0 = -\frac{1}{6}h^{il} h^{jm} c_{lm}^k e_i\wedge e_j \wedge
e_k,
\end{equation}
where $c_{ij}^k$ are the structure constants of $\mathfrak{g}$ with
respect to a basis $e_1,\ldots,e_n \in \mathfrak{g}$. Note that
$\star$ in this case is of ``Weyl-type'', meaning the star commutator
of elements $\xi,\zeta \in \Anti^\bullet\mathfrak{g}$ contains only
terms of odd $\lambda$-order (see also
Lemma~\ref{lemma:Fedosov-Weyl-Typ}).  Hence for $\Theta = \Theta_0 +
\lambda \Theta_1 + \cdots \in \Rothstein_-^\bullet(\Eps)[[\lambda]]$
we have $C_2(\Theta_0,\Theta_0) = C_2(\Theta_0,\Theta_1) +
C_2(\Theta_1,\Theta_0) = 0$, and Equation \eqref{eq:Theta-star-Theta}
reads as
\[
\Theta \star \Theta = \lambda^2
\RothBracket{\Theta_0,\Theta_1} + \lambda^3
\big(\delta_{\Theta_0} \Theta_2 + \tfrac{1}{2}\RothBracket{ \Theta_1,\Theta_1} +
C_3(\Theta_0,\Theta_0) \big) +\cdots.
\]
In order to achieve $\Theta\star\Theta = 0$ up to order two,
$\Theta_1$ has to satisfy $\RothBracket{\Theta_0,\Theta_1}= 0$, i.e.\
$m(\Theta_1,\cdot) = \ad(\Theta_1) = 0$. But for a semi-simple Lie
algebra this implies that $\Theta_1 = 0$. Hence in the third
$\lambda$-order of $\Theta\star\Theta$ only remains
$\delta_{\Theta_0}\Theta_2 + C_3(\Theta_0,\Theta_0)$. But a little
computation shows that
\[
C_3(\Theta_0,\Theta_0) = \frac{1}{3! 2^3} h^{il} c_{ik}^m c_{lm}^k =
\frac{1}{48} h^{il} h_{il} = \frac{1}{48}\delta_i^i =
\frac{\dim(\mathfrak{g})}{48},
\]
where we used that $h_{ij} = \mathrm{tr}(\ad(e_i)\circ \ad(e_j)) =
c_{ik}^l c_{jl}^k$ for the Killing form.  Since
$\delta_{\Theta_0}\Theta_2 $ has no contribution in $\DegR$-degree
zero, it follows that $\mathfrak{g}$ admits no quantization in the
sense of Definition~\ref{def:quantisierungI}.  

It therefore seems that our Definition~\ref{def:quantisierungI} for
the quantization of Courant algebroids is too restrictive. Hence the
following weaker definition might be more useful.
\begin{definition}[Quantization of Courant Algebroids, second
   version]\label{def:quantisierungII}
   Let $\Eps$ be a Courant algebroid over $\A$ with Courant structure
   $m \in \MC^3(\Eps)$, and let $\Theta_0\in \Rothstein^3(\Eps)$ be the
   associated homological element in the Rothstein-Poisson algebra. A
   quantization of the Courant structure on $\Eps$ is a formal series
   $\Theta = \Theta_0 +\lambda \Theta_1 + \ldots \in
   \Rothstein_-^\bullet(\Eps)[[\lambda]]$ such that $\Theta
   \star_\kappa \Theta$ is central, i.e such that
   $\ad_{\star_\kappa}(\Theta\star_\kappa \Theta) = [\Theta\star_\kappa
   \Theta,\,\cdot\,]_{\star_\kappa} = 0$.
\end{definition}

For a quantization $\Theta \in
\Rothstein_{-}^\bullet(\Eps)[[\lambda]]$ in this sense the map
$\ad_{\star_\kappa}(\Theta):\Rothstein^\bullet(\Eps)[[\lambda]]
\longrightarrow \Rothstein^\bullet(\Eps)[[\lambda]]$ has square zero
and defines a deformation of the differential $\delta_{\Theta_0} =
\RothBracket{\Theta_0,\cdot}$.  Moreover, the star product is also
well-defined in the cohomology of the deformed differential, since by
associativity of $\star_\kappa$ the map $\frac{1}{\lambda}
\ad_{\star_\kappa}(\Theta)$ is a graded derivation of $\star_\kappa$.

Let us consider again Example~\ref{example:semisimpleLie}, i.e.~$\Eps
= \mathfrak{g}$ is a finite-dimensional, semisimple Lie algebra with
Killing form, and $\Theta \in \Anti^3 \Eps$ the corresponding element
in the Rothstein algebra as defined in
Equation~(\ref{eq:Theta_for_LieAlg}).  Then $\Theta\star\Theta =
\lambda^3
C_3(\Theta,\Theta) \in \mathbb{k}[[\lambda]]$ and hence
$\ad_\star(\Theta\star\Theta) = 0$. For $\phi \in \Anti^k \Eps$ we
further get that
\begin{align*}
   \frac 1 \lambda \ad_{\star}(\Theta)\phi &= \frac{1}{\lambda} \big( \wedge \circ
   \E^{\lambda \calS} (\Theta \otimes \phi) - (-1)^{3k}\wedge
   \circ \E^{\lambda
     \calS} (\phi  \otimes \Theta) \big)\\
   & = \frac 1 \lambda \sum_{n=0}^\infty \frac{\lambda^n}{n!}
   \big(\wedge \circ \calS^n (\Theta
   \otimes \phi) - (-1)^{3k}  \wedge \circ \calS^n (\phi \otimes
   \Theta)\big)\\
   &= \frac 1 \lambda \sum_{n=0}^\infty \frac{\lambda^n}{n!}
   \big(\wedge \circ \calS^n (\Theta
   \otimes \phi) - (-1)^n  \wedge \circ \calS^n (\Theta \otimes \phi)\big)\\
     % & = \frac 1 \lambda \sum_{n=0}^\infty \lambda^n \wedge( \circ
     % \calS^n (\Thet%a
     % \otimes \phi - (-1)^{n} \circ \calS^n (\Theta \otimes \phi)))\\
   & = 2 \wedge \circ  \sum_{n=0}^\infty \frac{\lambda^{2n}}{(2n+1)!}
   \calS^{2n+1}(\Theta \otimes \phi)\\
   &  = \RothBracket{\Theta,\phi} + \frac{\lambda^2 }{3}
   \calS^3(\Theta \otimes \phi).
\end{align*}

\begin{example}[$\mathfrak{so}(3)$]
Let $\Eps = \mathbb{R}^3$ with the standard scalar product, and let
$\Theta = - e_1 \wedge e_2 \wedge e_3 \in \Rothstein^3(\Eps) =
\Anti^3 \Eps$. Then $\RothBracket{\Theta,\Theta} = 0$ and we find for
the derived bracket
\begin{align*}
   \RothBracket{\RothBracket{e_1,\Theta},e_2} &= \RothBracket{-e_2\wedge
   e_3,e_2} = e_3 \\
   \RothBracket{\RothBracket{e_2,\Theta},e_3} &= \RothBracket{e_1\wedge
   e_3,e_3} = e_1 \\
   \RothBracket{\RothBracket{e_3,\Theta},e_1} &= \RothBracket{-e_1\wedge
   e_2,e_1} = e_2,
\end{align*}
hence $\mathbb{R}^3$ together with the derived bracket is the Lie algebra
$\mathfrak{so}(3)$. For the deformed differential $\frac 1 \lambda
\ad_{\star}(\Theta)\phi = \RothBracket{\Theta,\phi} + \frac{\lambda^2
}{3} \calS^3(\Theta \otimes \phi)$ we get for $x,y,z \in
\mathbb{R}^3$ that
\begin{align*}
   \frac 1 \lambda \ad_{\star} (\Theta)x &= 
   \RothBracket{-e_1\wedge e_2 \wedge e_3, x} = -x^1
   e_2\wedge e_3 - x^2 e_3 \wedge e_1 - x^3 e_1
   \wedge e_2,\\
   \frac 1 \lambda \ad_{\star} (\Theta)(x\wedge y) &= 
   \RothBracket{-e_1\wedge e_2 \wedge
   e_3 , x \wedge y} \\
   &= (-x^1 e_2\wedge e_3 - x^2 e_3 \wedge e_1 - x^3 e_1
   \wedge e_2)\wedge y\\
   &\qquad\qquad- x \wedge (-y^1 e_2\wedge e_3 - y^2 e_3 \wedge e_1 -
   y^3 e_1
   \wedge e_2)\\
   &= (-x^1 y^1  - x^2 y^2  - x^3 y^3 )e_1\wedge e_2\wedge e_3\\
   &\qquad\qquad - (-x^1 y^1 - x^2 y^2 - x^3 y^3) e_1
   \wedge e_2 \wedge e_3)\\
   &= 0,
   \intertext{and}
   \frac 1 \lambda \ad_{\star} (\Theta)(x\wedge y \wedge z) &= \frac
   {\lambda^2} 3  \calS^3 (-e_1\wedge e_2 \wedge
   e_3)(x\wedge y
   \wedge z) = \frac{\lambda^2}{4} \det(x,y,z).
\end{align*}
This leads  now to the following diagram, where $\mathfrak{g} =
\mathfrak{so}(3)$,  $C^n =
\bigoplus_{2p+k=n} \lambda^p \Anti^k \mathfrak{g}$ and $d = \frac 1
\lambda \ad_{\star}(\Theta)$.
\begin{equation*}
\xymatrix{ C^0 \ar@{}[d]|{\parallel} \ar[r]^d & C^1
  \ar@{}[d]|{\parallel} \ar[r]^d & C^2 \ar@{}[d]|{\parallel} \ar[r]^d
  & C^3 \ar@{}[d]|{\parallel} \ar[r]^d & C^4 \ar@{}[d]|{\parallel}
  \ar[r]^d & C^5 \ar@{}[d]|{\parallel} \ar[r]^d & C^6
  \ar@{}[d]|{\parallel} \ar[r]^d & C^7 \ar@{}[d]|{\parallel} \\
  \mathbb{R}\ar[r]^{0} & \mathfrak{g}\ar[r]^(0.4){\sim} & \Anti^2
  \mathfrak{g}\ar[r]^0 \ar@{}[d]|{\bigoplus} & \Anti^3 \mathfrak{g}
  \ar@{}[d]|{\bigoplus} \ar[ddr]^(0.32){\sim} | \hole
  \ar[r]^0 & 0 \ar@{}[d]|{\bigoplus}  &&&\\
  & & \lambda \mathbb{R} \ar[r]^0 & \lambda \mathfrak{g}
  \ar[r]^(0.26){\sim} & \lambda \Anti^2 \mathfrak{g} \ar[r]^0
  \ar@{}[d]|{\bigoplus} & \lambda \Anti^3 \mathfrak{g}
  \ar[ddr]^(0.32){\sim} | \hole
  \ar@{}[d]|{\bigoplus} %\ar[ddr]^(0.32){\sim} | \hole
  \ar[r]^0 & 0 \ar@{}[d]|{\bigoplus}  \\
  & & & & \lambda^2 \mathbb{R} \ar[r]^0 & \lambda^2 \mathfrak{g}
  \ar[r]^(0.26)\sim &\lambda^2 \Anti^2 \mathfrak{g} \ar[r]^0
  \ar@{}[d]|{\bigoplus} & \lambda^2 \Anti^3 \mathfrak{g}
  \ar@{}[d]|{\bigoplus} \\
  & & & & & & \lambda^3 \mathbb{R} \ar[r]^0 & \lambda^3 \mathfrak{g}
  }
\end{equation*}
It follows for the cohomology of this complex that
\begin{align*}
   H^0 &= 0, & H^1 &= 0, & H^2 &= \lambda \mathbb{R}, & H^3 &= 0, & H^4 &= 0,\\
   H^5 &= \lambda H^3 = 0, & H^6 &= \lambda H^4 = 0, & \ldots
\end{align*}
Note the difference to the Lie algebra cohomology of
$\mathfrak{so}(3)$ defined via the Chevalley-Eilenberg differential,
where we have
\begin{align*}
   H^1_{CE}(\mathfrak{so}(3)) &= 0, &  H^2_{CE}(\mathfrak{so}(3)) &= 0, &
   H^3_{CE}(\mathfrak{so}(3)) &= \mathbb{R}, &
   H^4_{CE}(\mathfrak{so}(3)) &= 0, & \ldots
\end{align*}
\end{example}
\begin{example}[$\mathfrak{sl}(2)$]
Let again $\Eps = \mathbb{R}^3$, but now with the scalar product defined
by the matrix
\[
\begin{pmatrix}
   0 & 0 & 1\\ 0 & 2 & 0 \\ 1 & 0 & 0
\end{pmatrix}.
\]
As before we set $\Theta = - e_1 \wedge e_2 \wedge e_3$, but note that
the Rothstein-Poisson bracket on $\Rothstein^\bullet(\Eps)$ differs
from the one in the example above since it depends on the scalar
product. In this case the derived bracket defines the Lie algebra
$\mathfrak{sl}(2)$. We further get
 for $x,y,z \in \mathbb{R}^3$ that
\begin{align*}
   \frac 1 \lambda \ad_{\star} (\Theta)x &= 
   \RothBracket{-e_1\wedge e_2 \wedge e_3, x} = -x^3
   e_2\wedge e_3 + 2 x^2 e_1 \wedge e_3 - x^1 e_1
   \wedge e_2,\\
   \frac 1 \lambda \ad_{\star} (\Theta)(x\wedge y) &= 
   \RothBracket{-e_1\wedge e_2 \wedge
   e_3 , x \wedge y} = 0\\
   \intertext{and}
   \frac 1 \lambda \ad_{\star} (\Theta)(x\wedge y \wedge z) &= \frac
   {\lambda^2} 3 \wedge\circ \calS^3 (-e_1\wedge e_2 \wedge
   e_3)(x\wedge y
   \wedge z) = \frac{\lambda^2}{2} \det(x,y,z).
\end{align*}
This gives us the same results for the cohomology classes as in the
example above.
\end{example}

%
% Lie-Rinehart Pairs
%

\section{An Example: Lie-Rinehart Pairs}
\label{section:Lie-Rinehart-Pairs}

We briefly recall the definition of a Lie-Rinehart pair (see e.g.\
\cite{rinehart:1963a, huebschmann:1990a}). Like in the previous
sections $\ring{R}$ denotes a commutative ring with $\mathbb{Q}
\subseteq \ring{R}$.
\begin{definition}[Lie-Rinehart pair]\label{definition:LieRinehartPairs}
   A Lie-Rinehart pair $(\A,\mathcal{G})$ is a module $\mathcal{G}$ over an
   associative and commutative $\ring{R}$-algebra $\A$, together with
   a Lie bracket $[\cdot,\cdot]$ on $\mathcal{G}$ and a Lie algebra
   homomorphism $\rho: \mathcal{G} \longrightarrow \Der(\A)$ such that
   \[[u, a v] = a [u,v] + \rho(u)a\, v\;\quad\text{ for all $u,v \in
     \mathcal{G}$ and $a\in \A.$}
   \]
\end{definition}

Let now $\LL$ be a finitely generated projective module over $\A$ and
$\LL'$ its dual module. Then $\LplusL$ is also finitely generated and
projective, and the bilinear form given by
\[
\SP{u+\alpha,v+\beta} = \alpha(v) + \beta(u)\qquad \text{for } u,v\in
\LL, \alpha,\beta \in \LL'
\]
is strongly non-degenerate. Choosing a connection on $\LL$ gives us
automatically a metric connection for $\LplusL$. For the following we
will assume that the bilinear form on $\LplusL$ is full. Note that for
the important case where $\LL$ is the space of sections of a Lie
algebroid fullness is automatically given. This preparation allows us
now to consider the graded Lie algebra $\mathcal{C}^\bullet(\LplusL)$
as defined in Proposition~\ref{proposition:CRN-Bracket}, as well as to
construct the Rothstein-Poisson bracket on
$\Rothstein^\bullet(\LplusL)$ as it is described in
Theorem~\ref{theorem:RothsteinPoissonBracket}.  We further introduce a
bigrading on $\Rothstein^\bullet(\LplusL)$ by
\[
  \Rothstein^{a,b}(\LplusL) = \bigoplus_{\substack{p + \mu = a\\ p + \nu
      = b}}\SDer^p(\A)\otimes_\A \Anti_\A^\mu \LL \otimes_\A \Anti_\A^\nu
  \LL',
\]
so
\[
  \Rothstein^r(\LplusL) = \bigoplus_{a + b = r} \Rothstein^{a,b}(\LplusL).
\]
The next proposition is an immediate consequence from the definition
of the Rothstein-Poisson bracket.
\begin{proposition}
  \begin{enumerate}
  \item The Rothstein-Poisson bracket on
    $\Rothstein^{\bullet,\bullet}(\LplusL)$ is of bidegree $(-1,-1)$,
    i.e.
    \begin{equation}
      \RothBracket{\Rothstein^{a,b}(\LplusL), \Rothstein^{c,d}(\LplusL}
      \subseteq \Rothstein^{a+c-1,b+d-1}(\LplusL) .
    \end{equation}
  \item
    The space $\Rothstein^{1,\bullet}(\LplusL) =
    \bigoplus_{a=0}^\infty\Rothstein^{1,a}(\LplusL) $
    is a graded Lie subalgebra of $\Rothstein^\bullet(\LplusL)$.
  \end{enumerate}
\end{proposition}
Consider now the map $\mathcal{J}: \Rothstein^\bullet(\LplusL)
\longrightarrow \MC^\bullet(\LplusL)$ as defined in
Section~\ref{sec:TheIsomorphism}. Since $\mathcal{J}$ is a morphism of
graded Lie algebras, the image of $\Rothstein^{1,\bullet}(\LplusL)$
under $\mathcal{J}$ is a graded Lie subalgebra of
$\MC^\bullet(\LplusL)$.

As the algebraic analog to the setting in
\cite{crainic.moerdijk:2004a} we consider for $r = 1, 2, 3, \ldots$
the space $\CM^r(\LL)$ defined as the subspace of
$\Hom_\ring{R}(\Anti_\ring{R}^{r} \LL,\LL)$ consisting of all maps $m$
such that there exists a map $\sigma_m \in \Hom_\A(\Anti_\A^{r-1}
\LL,\Der(\A)) \cong \Anti_\A^{r-1}\LL' \otimes_\A\Der(\A)$ satisfying
\begin{equation}
   m(x_1,\ldots,x_{r-1}, a x_r) =   a m(x_1,\ldots,x_{r-1}, x_r) +
   \sigma_m(x_1,\ldots,x_{r-1}) a \;  x_r
\end{equation}
for all $x_1,\ldots,x_r \in \Eps$ and $a\in\A$.  Note that we require
here $\sigma_m$ to be $\A$-multilinear.  We finally set $\CM^0(\LL) =
\LL$ and $\CM^\bullet(\LL) = \bigoplus_{r=0}^\infty \CM^r(\LL)$.
\begin{proposition}
   We have an isomorphism
  \begin{equation}
     \mathcal{J}_\LL : \Rothstein^{1,\bullet}(\LplusL)
     \longrightarrow \CM^\bullet(\LL). 
  \end{equation}
  of graded modules.
\end{proposition}
\begin{proof}
   If $\phi \in \Rothstein^{1,\bullet}(\LplusL)$, the grading of the
   Rothstein-Poisson bracket implies that
   $\mathcal{J}(\phi)(u_1,\ldots,u_{r-1}) \in \LL$ for all
   $u_1,\ldots,u_{r-1} \in \LL$, hence we can restrict
   $\mathcal{J}(\phi)$ to an $\ring{R}$-multilinear map on
   $\LL$. Using that $\LL$ is an isotropic subspace of $\LplusL$, the
   definition of $\MC^\bullet(\LplusL)$  shows that the restriction to
   $\LL$ defines an element in $\CM^\bullet(\LL)$. This defines us an
   module homomorphism  $\mathcal{J}_\LL$, which in fact is an isomorphism
   where the inverse can be read off the proof of \cite[Lemma
   1]{crainic.moerdijk:2004a}: For $m\in \CM^r(\LL)$ consider the map
   $\mathrm{L}_m \in \LL \otimes_\A \Anti_\A^r \LL'$ given by
   \[
   \mathrm{L}_m(x_1,\ldots,x_r) = m(x_1,\ldots,x_r) -
   \sum_{i=1}^r (-1)^{r+i}
   \nabla_{\sigma_m(x_1,\ldots,\hat{x}_i,\ldots,x_r)} x_i
   \]
   for $x_1,\ldots,x_r\in \LL$, then $\mathcal{J}_\LL^{-1}(m) =
   (-1)^{\frac{r(r-1)}{2}}\big( \mathrm{L}_m + (-1)^r \sigma_m\big)
   \in \Rothstein^{1,r}(\LplusL) $. 
\end{proof}
Since $\Rothstein^{1,\bullet}(\LplusL)$ is a graded Lie subalgebra,
the isomorphism $\mathcal{J}_\LL$ defines a graded Lie bracket on
$\CM^\bullet(\LL)$. This bracket is the restriction of the bracket on
$\MC^\bullet(\LplusL)$ to $\LL$ and hence specified by $[x,y] = 0$ and
$[x,\mu] = i_x \mu$ for $x,y\in \LL$, $\mu \in \CM^r(\LL)$, and by the
recursion rule $[[\mu,\eta],x] = (-1)^r [[\mu,x],\eta] +
[\mu,[\eta,x]]$ for $\mu \in \CM^r(\LL)$, $\eta \in \CM^\bullet(\LL)$
and $x\in \LL$. One easily shows now  by induction that this bracket
coincides with the Nijenhuis-Richardson bracket
\cite{nijenhuis.richardson:1967a} restricted to $\CM^\bullet(\LL)$,
and hence (up to a sign) with the bracket defined in
\cite{crainic.moerdijk:2004a}. We therefore have:
\begin{corollary}\label{cor:J_L_is_Iso}
   $\CM^r(\LL)$ is closed under the Nijenhuis-Richardson bracket
   $[\cdot,\cdot]_\mathrm{NR}$ and $\mathcal{J}_\LL :
   \Rothstein^{1,\bullet}(\LplusL) \longrightarrow \CM^{\bullet}(\LL)$
   is an isomorphism of graded Lie algebras.
\end{corollary}

Now Lie-Rinehart structures on $\LL$ are in one-to-one correspondence
with homological elements $m\in \CM^2(\LL)$, i.e.\ elements $m\in
\CM^2(\LL)$ satisfying $[m,m]_\mathrm{NR} = 0$. As an immediate
consequence of Corollary~\ref{cor:J_L_is_Iso} we hence get the
following corollary.
\begin{corollary}
   Let $\mu \in \Rothstein^{1,2}(\LplusL)$. Then
   $\mathcal{J}_\LL(\mu)$ defines a Lie-Rinehart structure on
   $\LL$ if and only if $\RothBracket{\mu,\mu} = 0$.
\end{corollary}
As usual we can consider for an homological element $\mu \in
\Rothstein^{1,2}(\LplusL)$  the differential
$\delta_\mu: \Rothstein^{1,\bullet}(\LplusL)\longrightarrow
\Rothstein^{1,\bullet+1}(\LplusL)$ given by $\delta_\mu =
\RothBracket{\mu,\cdot}$, as well as the complex
\begin{equation*}
   \label{eq:TheComplex_LieRinehart}
   \xymatrix{  \LL
     \ar[r]^-{\delta_\mu} & \Rothstein^{1,1}(\LplusL)
     \ar[r]^-{\delta_\mu} &  \Rothstein^{1,2}(\LplusL)
     \ar[r]^-{\delta_\mu} &  \Rothstein^{1,3}(\LplusL)
     \ar[r]^-{\delta_\mu} &  }
   \cdots,
\end{equation*}
and its cohomology $H^\bullet_{CM,\mu}(\LL)$.  The second and third
cohomology classes are relevant for the deformation theory of
Lie-Rinehart pairs.

\begin{remark}
   The differential $\delta_\mu$ defines also a map $\delta_\mu:
   \Rothstein^{0,\bullet}(\LplusL)\longrightarrow
   \Rothstein^{0,\bullet+1}(\LplusL)$. But $\Rothstein^{0,k}(\LplusL) =
   \Anti_\A^k \LL'$, and a little computation shows that $\delta_\mu$
   restricted to $\Anti_\A^\bullet \LL'$ is the differential $\dif_{\LL}$
   of the Lie-Rinehart pair $(\A,\LL)$.  The cohomology of the
   associated complex
   \begin{equation*}
      \label{eq:TheComplex_LieRinehart_dif}
      \xymatrix{  \A  \ar[r]^-{\delta_\mu} & \LL'
        \ar[r]^-{\delta_\mu} & \Anti_\A^2 \LL'
        \ar[r]^-{\delta_\mu} &  \Anti_\A^3 \LL'
        \ar[r]^-{\delta_\mu} &  \Anti_\A^4 \LL'
        \ar[r]^-{\delta_\mu} &  }
      \cdots.
   \end{equation*}
   is then the Lie algebroid cohomology of the Lie-Rinehart pair $(\A,\LL)$.
\end{remark}
\begin{remark}
   The composition $\CM^\bullet(\LL) \longrightarrow
   \Rothstein^\bullet(\LplusL) \longrightarrow
   \MC^\bullet_{\LL}(\LplusL) $
   defines the construction of a Courant structure on $\LplusL$ for a
   given Lie-Rinehart structure on $\LL$, and it is not a surprise
   that this leads to the well known bracket given by
   \begin{equation}
      [(u,\alpha),(v,\beta)] = ([u,v], \dif_{\LL} i_u \beta + i_u
      \dif_\LL \beta - i_v \dif_{\LL} \alpha)
   \end{equation}
   for $u,v \in \LL$ and $\alpha,\beta \in \LL'$.  If especially $\LL =
   \Gamma^\infty(TM)$ for a smooth manifold $M$, then this yields the
   usual Courant bracket on $TM\oplus T^\ast M$.
\end{remark}

Suppose for the following that $\Der(\A)$ is finitely generated and
projective, and let $\star_\kappa$ be a Fedosov star product for
$\Rothstein^{\bullet,\bullet}(\LplusL)$ as defined in
Section~\ref{section:Fedosov}. A slightly more detailed consideration
of the involved degrees shows that the star product in this situation
also respects the bigrading, meaning that for $\phi \in
\Rothstein^{a,b}(\LplusL)$ and $\psi \in \Rothstein^{k,l}(\LplusL)$
one has
\[
\phi\star_\kappa \psi \in \bigoplus_{n\in \mathbb{N}}
\lambda^n \Rothstein^{a+k-n,b+l-n}(\LplusL).
\]
Let $\mu \in \Rothstein^{1,2}(\LplusL)$ define a Lie-Rinehart
structure on $\LL$. So $\RothBracket{\mu,\mu} = 0$ and hence
\[
\tfrac 1 2\,\mu \star_\kappa \mu = \lambda^2 C_2(\mu,\mu) +  \lambda^3
C_3(\mu,\mu) + \ldots
\]
However, due to the bihomogeneity of the star product it follows that
already the third order term $C_3(\mu,\mu) \in
\Rothstein^{-1,1}(\LplusL) = \{0\}$ vanishes, so there remains only
the second order term. In particular, if we choose $\kappa =
\frac{1}{2}$, then Lemma~\ref{lemma:Fedosov-Weyl-Typ} implies that the
second order term vanishes as well. In this case it therefore follows
that $\mu \star_{\scriptscriptstyle{\frac{1}{2}}} \mu = 0$ without any
higher order corrections. A similar result was obtained in \cite[Satz
4.1]{eilks:2004a}, but there for the standard ordered
star product, i.e.\ for the case $\kappa = 0$.
\begin{theorem}
    \label{theorem:LieRinehartQuantizable}%
    The Courant algebroid structure arising from a Lie-Rinehart
    structure allows for a deformation quantization in the sense of
    Definition~\ref{def:quantisierungI}.
\end{theorem}

%
% references
%
\section*{Bibliography}
\begin{footnotesize}
    \renewcommand{\arraystretch}{0.5}
%    \bibliographystyle{ewde}
%    \bibliography{dqarticle,dqbook,dqprocentry,dqproceeding,dqthesis,misc,preprints,courant}

\begin{thebibliography}{10}

\bibitem {balavoine:1995a}
{\sc Balavoine, D.: }\newblock {\em Deformations of algebras over a quadratic
  operad}.
\newblock In: {\em Operads: Proceedings of Renaissance Conferences (Hartford,
  CT/Luminy, 1995)}, vol. 202 in {\em Contemp. Math.},   207--234, Providence,
  RI, 1997. Amer. Math. Soc.

\bibitem {bordemann:1996a}
{\sc Bordemann, M.: }\newblock {\em On the deformation quantization of
  super-Poisson brackets}.
\newblock Preprint Freiburg FR-THEP-96/8  {\bf q-alg/9605038} (May 1996).

\bibitem {bordemann:2000a}
{\sc Bordemann, M.: }\newblock {\em The deformation quantization of certain
  super-{P}oisson brackets and {BRST} cohomology}.
\newblock In: {\sc Dito, G., Sternheimer, D. (eds.): }\newblock {\em
  Conf{\'e}rence Mosh{\'e} Flato 1999. Quantization, Deformations, and
  Symmetries}, {\em Mathematical Physics Studies} no. { 22},   45--68. Kluwer
  Academic Publishers, Dordrecht, Boston, London, 2000.

\bibitem {bursztyn.cavalcanti.gualtieri:2007a}
{\sc Bursztyn, H., Cavalcanti, G.~R., Gualtieri, M.: }\newblock {\em Reduction
  of {C}ourant algebroids and generalized complex structures}.
\newblock Adv. Math.  {\bf 211}.2 (2007), 726--765.

\bibitem {bursztyn.waldmann:2005b}
{\sc Bursztyn, H., Waldmann, S.: }\newblock {\em Completely positive inner
  products and strong {M}orita equivalence}.
\newblock Pacific J. Math.  {\bf 222} (2005), 201--236.

\bibitem {courant:1990a}
{\sc Courant, T.~J.: }\newblock {\em {D}irac Manifolds}.
\newblock Trans. AMS  {\bf 319}.2 (1990), 631--661.

\bibitem {crainic.moerdijk:2004a}
{\sc Crainic, M., Moerdijk, I.: }\newblock {\em {D}eformations of {L}ie
  Brackets: Cohomologial Aspects}.
\newblock preprint math.DG/0403434, 2004.

\bibitem {eilks:2004a}
{\sc Eilks, C.: }\newblock {\em {BRST}-{R}eduktion linearer {Z}wangsbedingungen
  im {R}ahmen der {D}e\-for\-ma\-tions\-quan\-ti\-sier\-ung}.
\newblock Master's thesis, Fakult{\"{a}}t f{\"{u}}r Physik,
  Albert-Ludwigs-Universit{\"{a}}t, Freiburg, 2004

\bibitem{fedosov:1996a}
{\sc Fedosov, B. V.:} \newblock
{\em Deformation quantization and index theory}.
Akademie Verlag, Berlin, 1996.

\bibitem {filippov:1985a}
{\sc Filippov, V.~T.: }\newblock {\em {$n$}-{L}ie algebras}.
\newblock Sibirsk. Mat. Zh.  {\bf 26}.6 (1985), 126--140, 191.

\bibitem {gerstenhaber:1964a}
{\sc Gerstenhaber, M.: }\newblock {\em On the Deformation of Rings and
  Algebras}.
\newblock Ann. Math.  {\bf 79} (1964), 59--103.

\bibitem {gualtieri:2003a}
{\sc Gualtieri, M.: }\newblock {\em Generalized complex geometry}.
\newblock PhD thesis, St John's College, University of Oxford, Oxford, 2003.
\newblock math.DG/0401221.

\bibitem {huebschmann:1990a}
{\sc Huebschmann, J.: }\newblock {\em Poisson cohomology and quantization}.
\newblock J. Reine Angew. Math.  {\bf 408} (1990), 57--113.

\bibitem {huebschmann:1998a}
{\sc Huebschmann, J.: }\newblock {\em Lie-Rinehart algebras, Gerstenhaber
  algebras, and Batalin-Vilkovisky algebras}.
\newblock Ann. Inst. Fourier  {\bf 48} (1998), 425--440.

\bibitem {keller:2004a}
{\sc Keller, F.: }\newblock {\em Deformation von {L}ie-{A}lgebroiden und
  {D}irac-{S}trukturen}.
\newblock Master's thesis, Fakult{\"{a}}t f{\"{u}}r Mathematik und Physik,
  Physikalisches Institut, Albert-Ludwigs-Universit{\"{a}}t, Freiburg, 2004.

\bibitem {keller.waldmann:2007a}
{\sc Keller, F., Waldmann, S.: }\newblock {\em Formal Deformations of {D}irac
  Structures}.
\newblock J. Geom. Phys.  {\bf 57} (2007), 1015--1036.

\bibitem {kontsevich:2003a}
{\sc Kontsevich, M.: }\newblock {\em Deformation Quantization of {P}oisson
  manifolds}.
\newblock Lett. Math. Phys.  {\bf 66} (2003), 157--216.

\bibitem {kosmann-schwarzbach:2004b}
{\sc Kosmann-Schwarzbach, Y.: }\newblock {\em Derived brackets}.
\newblock Lett. Math. Phys.  {\bf 69} (2004), 61--87.

\bibitem {neumaier:1998a}
{\sc Neumaier, N.: }\newblock {\em Sternprodukte auf {K}otangentenb{\"{u}}ndeln
  und {O}rdnungs-{V}orschriften}.
\newblock Master's thesis, Fakult{\"{a}}t f{\"{u}}r Physik,
  Albert-Ludwigs-Universit{\"{a}}t, Freiburg, 1998.
\newblock Available at \texttt{http://idefix.physik.uni-freiburg.de/\~{}nine/}.


\bibitem {nijenhuis.richardson:1967a}
{\sc Nijenhuis, A., Richardson, Jr. , R.~W.: }\newblock {\em Deformations of
  {L}ie algebra structures}.
\newblock J. Math. Mech.  {\bf 17} (1967), 89--105.

\bibitem {rinehart:1963a}
{\sc Rinehart, G.: }\newblock {\em Differential forms on general commutative
  algebras}.
\newblock Trans. Amer. Math. Soc.  {\bf 108} (1963), 195--222.

\bibitem {rothstein:1991a}
{\sc Rothstein, M.: }\newblock {\em The structure of supersymplectic
  supermanifolds}.
\newblock In: {\sc Bartocci, C., Bruzzo, U., Cianci, R. (eds.): }\newblock {\em
  Differential geometric methods in theoretical physics (Rapallo, 1990)}, vol.
  375 in {\em Lecture Notes in Physics},   331--343. Springer, Berlin, 1991.
\newblock Proceedings of the Nineteenth International Conference held in
  Rapallo, June 19--24, 1990.

\bibitem {rotkiewicz:2005a}
{\sc Rotkiewicz, M.: }\newblock {\em Cohomology ring of {$n$}-{L}ie algebras}.
\newblock Extracta Math.  {\bf 20}.3 (2005), 219--232.

\bibitem {roytenberg:1999a}
{\sc Roytenberg, D.: }\newblock {\em Courant Algebroids, derived brackets and
  even symplectic supermanifolds}.
\newblock PhD thesis, UC Berkeley, Berkeley, 1999.
\newblock math.DG/9910078.

\bibitem {roytenberg:2002b}
{\sc Roytenberg, D.: }\newblock {\em On the structure of graded symplectic
  supermanifolds and {C}ourant algebroids}.
\newblock In: {\sc Voronov, T. (eds.): }\newblock {\em Quantization, Poisson
  brackets and beyond (Manchester, 2001)}, vol. 315 in {\em Contemporary
  Mathematics},   169--185. American Mathematical Society, Providence, RI,
  2002.
\newblock Papers from the London Mathematical Society Regional Meeting held
  July 6, 2001 and the Workshop on Quantization, Deformations, and New
  Homological and Categorical Methods in Mathematical Physics held at the
  University of Manchester, Manchester, July 7--13, 2001.

\bibitem {roytenberg.weinstein:1998a}
{\sc Roytenberg, D., Weinstein, A.: }\newblock {\em Courant Algebroids and
  Strongly Homotopy Lie Algebras}.
\newblock Lett. Math. Phys.  {\bf 46} (1998), 81--93.

\bibitem {uchino:2002a}
{\sc Uchino, K.: }\newblock {\em Remarks on the definition of a {C}ourant
  algebroid}.
\newblock Lett. Math. Phys.  {\bf 60}.2 (2002), 171--175.

\bibitem {vinogradov.vinogradov:1998a}
{\sc Vinogradov, A., Vinogradov, M.: }\newblock {\em On multiple
  generalizations of {L}ie algebras and {P}oisson manifolds}.
\newblock In: {\em Secondary calculus and cohomological physics (Moscow,
  1997)}, vol. 219 in {\em Contemp. Math.},   273--287. Amer. Math. Soc.,
  Providence, RI, 1998.

\bibitem {waldmann:2007a}
{\sc Waldmann, S.: }\newblock {\em Poisson-{G}eometrie und
  {D}eformationsquantisierung. {E}ine {E}inf{\"u}hrung}.
\newblock Springer-Verlag, Heidelberg, Berlin, New York, 2007.


\end{thebibliography}

\end{footnotesize}

\end{document}